\documentclass[12pt]{article}

\usepackage[T1]{fontenc}
\usepackage[english]{babel}
\usepackage{amssymb}
\usepackage{amsmath}
\usepackage{epsfig}
\usepackage{color}
\usepackage{hyperref}
\usepackage[a4paper]{geometry}
 \usepackage{cancel}

\allowdisplaybreaks

\newcommand{\avg}[1]{\langle{#1}\rangle}
\renewcommand\r{\mathbb{R}}

\newcommand\1{{\bf 1}}

\newtheorem{theorem}{Theorem}[section]
\newtheorem{lemma}[theorem]{Lemma}
\newtheorem{proposition}[theorem]{Proposition}

\newenvironment{proof}[1][Proof]{\begin{trivlist}
\item[\hskip \labelsep {\bfseries #1}] \it }{$\blacksquare$\end{trivlist}}

\newtheorem{remark}[theorem]{Remark}
\graphicspath{{Figures/}}

\definecolor{red}{rgb}{1,0,0}
\definecolor{blue}{rgb}{0,0,1}

\begin{document}

\title{Layer-averaged Euler and Navier-Stokes equations}

\author{M.-O.~Bristeau, C.~Guichard, B.~Di~Martino, J.~Sainte-Marie}
\date{\today}

\maketitle
\begin{abstract}
In this paper we propose a strategy to approximate incompressible hydrostatic free surface Euler
and Navier-Stokes models. The main advantage of the proposed models is
that the water depth is a dynamical variable of the system and hence
the model is formulated over a fixed domain.

The proposed strategy extends previous works approximating the Euler
and Navier-Stokes systems using a multilayer description. Here, the
needed closure relations are obtained using an energy-based optimality
criterion instead of an asymptotic expansion. Moreover, the
layer-averaged description is successfully applied to the
Navier-Stokes system with a general form of the Cauchy stress tensor.
\end{abstract}

{\it Keywords~:} Incompressible Navier-Stokes equations,
incompressible Euler equations, free surface flows, newtonian fluids,
complex rheology


\tableofcontents

\section{Introduction}

Due to computational issues associated with the free surface Navier-Stokes or
Euler equations, the simulations of geophysical flows are often carried
out with shallow water type
models of reduced complexity. Indeed, for vertically
averaged models such as the Saint-Venant system \cite{saint-venant},
efficient and robust numerical techniques
(relaxation schemes \cite{bouchut}, kinetic schemes
\cite{perthame,JSM_entro},\ldots) are available and avoid to deal with moving meshes.

In order to describe and simulate complex flows where the velocity
field cannot be approximated by its vertical mean, multilayer models
have been developed~\cite{audusse,bristeau3,bristeau2,bouchut1,pares,pares1}.
Unfortunately these models are physically relevant for non miscible fluids.

In~\cite{nieto_chacon,JSM_M2AN,JSM_JCP,JSM_M3AS}, some authors have proposed
a simpler and more general formulation for multilayer model with mass
exchanges between the layers. The obtained model has the form of a
conservation law with source terms, its hyperbolicity remains an
open question. Notice that in~\cite{JSM_JCP} the hydrostatic Navier-Stokes
equations with variable density is tackled and in~\cite{JSM_M3AS} the
approximation of the non-hydrostatic terms in the multilayer context
is studied. With respect to commonly used Navier--Stokes solvers,
 the appealing features of the proposed multilayer approach 
  are the easy handling of the free surface,
 which does not require moving meshes (e.g.~\cite{ale_telemac}), and the possibility to
 take advantage of robust and  accurate numerical techniques
 developed in extensive amount for classical one-layer Saint Venant equations. 
Recently, the multilayer model developed in  \cite{nieto_chacon}
has been adapted in  \cite{nieto_diaz_mangeney_reina} 
in the case of the $\mu$(I)-rheology through an asymptotic analysis.

The objective of the paper is twofold. First we want to present another derivation of the
models proposed in~\cite{JSM_M2AN,JSM_JCP,JSM_M3AS}, no more
based on an asymptotic expansion but on an energy-based optimality
criterion. Such a strategy is widely used in the kinetic framework to obtain
kinetic descriptions e.g. of conservations
laws~\cite{levermore,perthame}. Second, we intend
to obtain a multilayer formulation of the Navier-Stokes system with a
rheology more complex than the one arising when considering newtonian fluids.

The paper is organized as follows. In Section~\ref{sec:NS} we recall
the incompressible hydrostatic Navier-Stokes equations with free surface with the
associated boundary conditions. In Section~\ref{sec:av_euler} we
detail the layer averaging process for the Euler system and obtained
the required closure relations. The proposed layer-averaged Euler
system is given in Section~\ref{sec:euler_mc} and its extension to the
Navier-Stokes system with a general rheology is presented in Section~\ref{sec:av_NS}.

\section{The Navier-Stokes system}
\label{sec:NS}

We consider the two-dimensional hydrostatic Navier-Stokes system \cite{lions}
describing a free surface gravitational flow moving over a bottom
topography $z_b(x)$. For free surface flows, the hydrostatic assumption consists in
neglecting the vertical acceleration, see \cite{brenier,grenier,masmoudi} for justifications
of the obtained models.

\subsection{The hydrostatic Navier-Stokes system}

We denote with $x$ and $z$ the horizontal and vertical directions, respectively.  
The system has the form:
\begin{eqnarray}
 \frac{\partial  u}{\partial x}+\frac{\partial w}{\partial z} & = & 0,\label{eq:NS_2d1}\\
\frac{\partial u}{\partial t } + \frac{\partial u^2}{\partial x } +\frac{\partial uw }{\partial z }+ 
\frac{\partial p}{\partial x }  & = & \frac{\partial \Sigma_{xx}}{\partial x} + \frac{\partial \Sigma_{xz}}{\partial z},\label{eq:NS_2d2}\\
\frac{\partial p}{\partial z}  & = & -g + \frac{\partial \Sigma_{zx}}{\partial x} + \frac{\partial \Sigma_{zz}}{\partial z},
\label{eq:NS_2d3}
\end{eqnarray}
and we consider solutions of the equations  for
$$t>t_0, \quad x \in \r, \quad z_b(x) \leq z \leq \eta(x,t),$$
where $\eta(x,t)$ represents the free surface elevation, ${\bf u}=(u,w)^T$ the velocity
vector, $p$ the fluid pressure and $g$ the gravity acceleration. The water depth
is $H = \eta - z_b$, see Fig.~\ref{fig:notations}. The Cauchy
stress tensor $\Sigma_T$  is defined by $\Sigma_T = - p I_d +\Sigma$ with
$$\Sigma = \left(\begin{array}{cc} \Sigma_{xx} & \Sigma_{xz}\\ 
\Sigma_{zx} & \Sigma_{zz}\end{array}\right),$$
and $\Sigma$ represents the fluid rheology.
\begin{figure}
\begin{center}
\includegraphics[height=5cm]{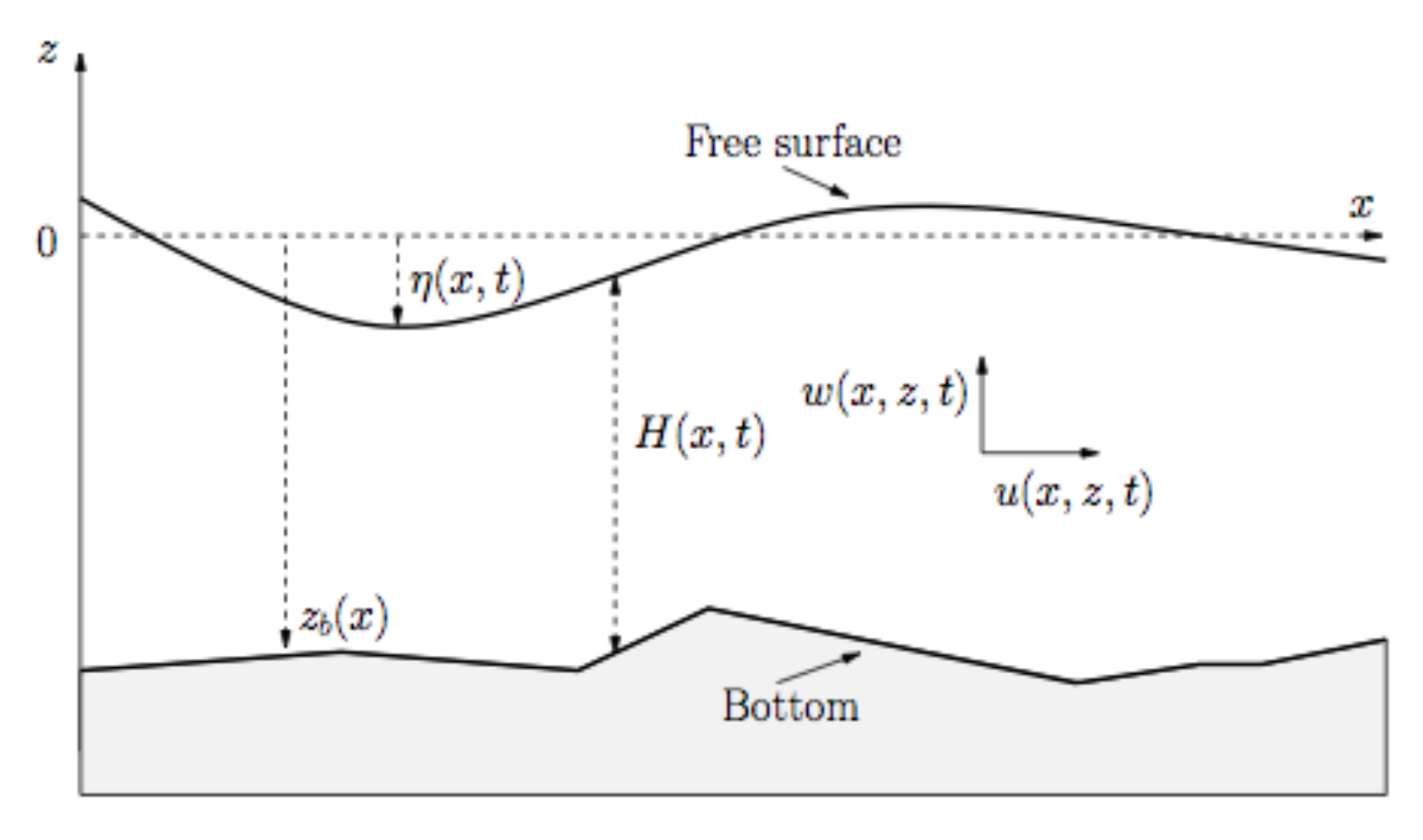}
\caption{Flow domain with water height $H(x,t)$, free surface $\eta(x,t)$ and bottom $z_b(x)$.}
\label{fig:notations}
\end{center}
\end{figure}

As in Ref.~\cite{gerbeau}, we introduce the indicator function for the fluid
region
\begin{equation}
\varphi(x,z,t) = \left\{\begin{array}{ll}
1 & \hbox{ for } (x,z) \in \Omega = \{(x,z)\,|\, z_b \leq z \leq \eta\},\\
0 & \hbox{ otherwise.}
\end{array}\right.
\label{eq:varphi}
\end{equation}
The fluid region is advected by the flow, which can be expressed, thanks
to the incompressibility condition, by the relation
\begin{equation}
\frac{\partial \varphi}{\partial t} + \frac{\partial \varphi u}{\partial
x} + \frac{\partial \varphi w}{\partial z} = 0.
\label{eq:advec_phi}
\end{equation}
The solution $\varphi$ of this equation takes the values 0 and 1 only
but it needs not be of the form (\ref{eq:varphi}) at all times. The
analysis below is limited to the conditions where this form is
preserved. For a more complete presentation of the Navier-Stokes
system and its closure, the reader can refer to~\cite{lions}.

\begin{remark}
Notice that in the fluid domain,
Eq.~(\ref{eq:advec_phi}) reduces to the divergence free condition whereas
across the upper and lower boundaries it gives the
kinematic boundary conditions defined in the following.
\end{remark}

\subsection{Boundary conditions}

The system (\ref{eq:NS_2d1})-(\ref{eq:NS_2d3}) is completed with boundary conditions. We not consider here lateral boundary conditions that can be usual
usual inflow and outflow boundary conditions.
The outward unit normal vector to the free surface ${\bf n}_s$ and the upward
unit normal vector to the bottom ${\bf n}_b$ are given by
$${\bf n}_s = \frac{1}{\sqrt{1 + \bigl(\frac{\partial \eta}{\partial x}\bigr)^2}}
  \left(\begin{array}{c} -\frac{\partial \eta}{\partial x}\\ 1 \end{array} \right), 
  \quad {\bf n}_b = 
  \frac{1}{\sqrt{1 + \bigl(\frac{\partial z_b}{\partial x}\bigr)^2}} 
   \left(\begin{array}{c} -\frac{\partial z_b}{\partial x}\\ 1 \end{array} \right) \equiv \left(\begin{array}{c} -s_b\\ c_b \end{array} \right),$$
 respectively. We use here the same definition for $s_b(x)$ and $c_b(x)$ as in \cite{bouchut}, $c_b(x)>0$ is the cosine of the angle between
 ${\bf n}_b$ and the vertical.

\subsubsection{Free surface conditions}

At the free surface we have the kinematic boundary condition
\begin{equation}
\frac{\partial \eta}{\partial t} + u_s \frac{\partial \eta}{\partial x}
-w_s = 0,
\label{eq:free_surf} 
\end{equation}
where the subscript $s$ indicates the value of the
considered quantity at the free surface.

Assuming negligible the air viscosity, the continuity
of stresses at the free boundary imposes
\begin{equation}
\Sigma_T {\bf n}_s = -p^a {\bf n}_s,
\label{eq:BC_h}
\end{equation}
where $p^a=p^a(x,t)$ is a given function corresponding to the
atmospheric pressure.
Within this paper, we consider $p^a=0$.

\subsubsection{Bottom conditions}

The kinematic boundary condition at the bottom 
consists in a classical no-penetration condition:
\begin{equation}
{\bf u}_b \cdot {\bf n}_b = 0,\quad 
\mbox{or}\quad u_b \frac{\partial z_b}{\partial x} - w_b = 0.
\label{eq:bottom} 
\end{equation}
For the stresses at the bottom we consider a wall law under the form 
\begin{equation}
\Sigma_T {\bf n}_b - ( {\bf n}_b \cdot \Sigma_T {\bf n}_b) {\bf n}_b = \kappa {\bf u_b} 
\label{eq:BC_z_b}
\end{equation}
and for ${\bf t}_b = ^t\!\!(c_b,s_b)$, using (\ref{eq:bottom}) we have
\begin{equation}
{\bf t}_b \cdot \Sigma_T {\bf n}_b =\frac{\kappa}{c_b}  u_b,
\label{eq:BC_z_b1}
\end{equation}
If $\kappa({\bf u_b},H)$ is constant then
we recover a Navier friction condition as in \cite{gerbeau}. Introducing
a laminar friction $k_l$ and a turbulent friction $k_t$, we use the expression 
$$\kappa({\bf u_b},H) = k_l + k_t H |{\bf u_b}|,$$
corresponding to the boundary condition used in \cite{marche}. Another
form of $\kappa({\bf u_b},H)$ is used in \cite{bouchut}, and for other
wall laws the reader can also refer
to \cite{valentin}. Due to thermo-mechanical considerations, in the
sequel we will suppose $\kappa({\bf u_b},H) \geq 0$, and $\kappa({\bf
u_b},H)$ will be  often simply denoted by $\kappa$.

\subsection{Other writing}

For reasons that will appear
later, we rewrite~\eqref{eq:NS_2d1}-\eqref{eq:NS_2d3} under the form
\begin{eqnarray}
& & \frac{\partial u}{\partial x}+\frac{\partial w}{\partial z} =0,\label{eq:NS_2d1_mod}\\
& & \frac{\partial u}{\partial t } + \frac{\partial u^2}{\partial x }
+\frac{\partial uw }{\partial z }+ g\frac{\partial \eta}{\partial x } =
\frac{\partial \Sigma_{xx}}{\partial x} + \frac{\partial
  \Sigma_{xz}}{\partial z} + \frac{\partial^2 }{\partial
  x^2}\int_z^\eta \Sigma_{zx} dz_1 - \frac{\partial \Sigma_{zz}}{\partial x},\label{eq:NS_2d2_mod}
\end{eqnarray}
where Eq.~\eqref{eq:NS_2d2_mod} has been obtained as
follows. Integrating Eq.~\eqref{eq:NS_2d3} from $z$ to $\eta$ and
taking into account the boundary condition~\eqref{eq:BC_h} gives
\begin{equation}
p = g(\eta - z) - \frac{\partial }{\partial
  x}\int_z^\eta \Sigma_{zx} dz_1 + \Sigma_{zz}.
\label{eq:p_exp}
\end{equation}
Inserting the previous expression for $p$ in Eq.~\eqref{eq:NS_2d2} gives Eq.~\eqref{eq:NS_2d2_mod}.

\subsection{Energy balance}

\begin{lemma}
We recall the fundamental stability property related to the fact that
the hydrostatic Navier-Stokes system admits an energy that can be
written under the form
\begin{multline}
\frac{\partial}{\partial t}\int_{z_b}^{\eta} E\ dz  +
\frac{\partial}{\partial x}\int_{z_b}^{\eta} \Bigl[   u \Bigl( E + g(\eta-z)
  - (\Sigma_{xx} - \Sigma_{zz}) - \frac{\partial}{\partial    x}\int_z^\eta \Sigma_{zx} dz_1 \Bigr) - w\Sigma_{zx} \Bigr]  dz\\
=  - \int_{z_b}^\eta  \Bigl( \frac{\partial u}{\partial x}(\Sigma_{xx}- \Sigma_{zz}) 
+ \frac{\partial u}{\partial z}  \Sigma_{xz}+  \frac{\partial w}{\partial x} \Sigma_{zx}  \Bigr) dz
- \frac{\kappa}{c_b^3} u_b^2,
\label{eq:energy_eq_mod}
\end{multline}
with
\begin{equation}
E = \frac{u^2}{2} + gz.
\label{eq:energy_exp}
\end{equation} 
\label{lem:lemma_energy}
\end{lemma}

\begin{proof}[Proof of lemma~\ref{lem:lemma_energy}]

The way the energy balance~\eqref{eq:energy_eq_mod} is obtained is
classical. Considering smooth solutions, first we multiply Eq.~\eqref{eq:NS_2d2} by $u$ and Eq.~\eqref{eq:NS_2d3} by
$w$ then we sum the two obtained equations. After simple manipulations
and using the kinematic and dynamic boundary conditions~\eqref{eq:free_surf}-\eqref{eq:BC_z_b}, we obtain the
relation
\begin{multline*}
\frac{\partial}{\partial t}\int_{z_b}^{\eta} E\ dz  +
\frac{\partial}{\partial x}\int_{z_b}^{\eta} \Bigl[ u\bigl( E + p
  \bigr) - u\Sigma_{xx} -w\Sigma_{zx}\Bigr] dz\\
=  - \int_{z_b}^\eta \Sigma_{xx}\frac{\partial u}{\partial x}  dz
-\int_{z_b}^\eta \Sigma_{xz}\frac{\partial u}{\partial z} dz -
\int_{z_b}^\eta \frac{\partial w}{\partial x} \Sigma_{zx} dz -
\int_{z_b}^\eta \Sigma_{zz}\frac{\partial w}{\partial z}  dz - \frac{\kappa}{c_b^3}
u_b^2.
\end{multline*}
By using Eq.~\eqref{eq:NS_2d1} and replacing $p$ by its expression given by~\eqref{eq:p_exp} in the previous
relation gives the result.

\end{proof}

\section{Depth-averaged solutions of the Euler system}
\label{sec:av_euler}

In this section, neglecting the viscous effects in Eqs.~\eqref{eq:NS_2d1}-\eqref{eq:NS_2d3}, we consider the
free surface hydrostatic Euler
equations written in a conservative form
\begin{eqnarray}
& & \frac{\partial \varphi}{\partial t}+\frac{\partial \varphi u}{\partial x}+\frac{\partial \varphi w}{\partial z} =0,\label{eq:euler_2d1}\\
& & \frac{\partial \varphi u}{\partial t } + \frac{\partial \varphi u^2}{\partial x } +\frac{\partial \varphi uw }{\partial z }+ \frac{\partial p}{\partial x } = 0,\label{eq:euler_2d2}\\
& & \hspace*{3cm} \frac{\partial p}{\partial z} = -
\varphi g,\label{eq:euler_2d3}
\end{eqnarray}
with $\varphi$ defined by~\eqref{eq:varphi}. This system is completed
with the boundary conditions~\eqref{eq:free_surf},\eqref{eq:bottom}
and~\eqref{eq:BC_h} that reduces to
\begin{equation}
p_s = 0.
\label{eq:ps}
\end{equation}
From Eqs.~\eqref{eq:euler_2d3},\eqref{eq:ps}, we get
\begin{equation}
p = \varphi g (\eta - z).
\label{eq:pres}
\end{equation}

The energy balance associated with the
hydrostatic Euler system is given by
\begin{equation}
\frac{\partial}{\partial t}\int_{z_b}^{\eta} E\ dz +
\frac{\partial}{\partial x}\int_{z_b}^{\eta} u\bigl( E + p
  \bigr) dz = 0,
\label{eq:energy_euler_eq}
\end{equation}
with $E$ defined by~\eqref{eq:energy_exp}.

\subsection{Vertical discretization of the fluid domain}

The interval $[z_b,\eta]$ is
divided into $N$ layers $\{L_\alpha\}_{\alpha\in\{1,\ldots,N\}}$ of
thickness $l_\alpha H(x,t)$ where each layer
$L_\alpha$ corresponds to the points satisfying $z \in L_\alpha(x,t) = ]z_{\alpha-1/2},z_{\alpha+1/2}[$
with
\begin{equation}
\left\{\begin{array}{l}
z_{\alpha+1/2}(x,t) = z_b(x) + \sum_{j=1}^\alpha l_j H(x,t),\\
h_\alpha(x,t) = z_{\alpha+1/2}(x,t) - z_{\alpha-1/2}(x,t)=l_\alpha H(x,t),
 \quad \alpha\in \{1,\ldots,N\},
\end{array}\right.
\label{eq:layer}
\end{equation}
 with $l_j>0, \quad \sum_{j=1}^N l_j=1$, see
Fig.~\ref{fig:notations_mc}.

We also define
\begin{equation}
z_\alpha = \frac{z_{\alpha+1/2} + z_{\alpha-1/2}}{2} =
z_{\alpha-1/2} + \frac{h_\alpha}{2}, \quad \alpha=\{1,\ldots,N\}.
\label{eq:zalpha}
\end{equation}

We finally introduced the distance between the midpoints of the layers,
\begin{equation}
h_{\alpha+1/2} = z_{\alpha+1} - z_{\alpha} = \frac{ h_{\alpha+1} + h_\alpha }{2}, \quad \alpha=\{1,\ldots,N-1\}.
\label{eq:halphademi}
\end{equation}

\begin{figure}[hbtp]
\begin{center}
\includegraphics[height=6cm]{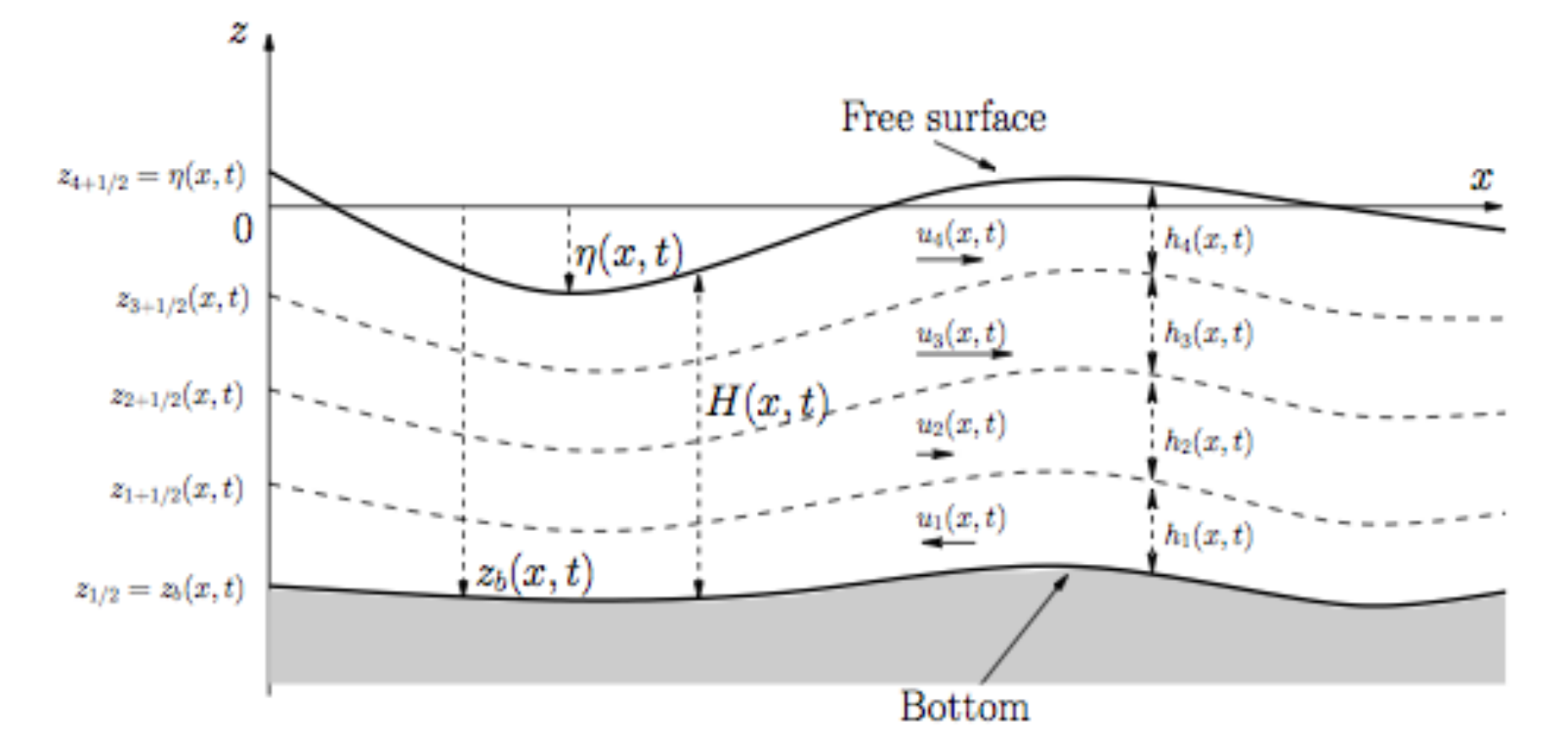}
\caption{Notations for the multilayer approach.}
\label{fig:notations_mc}
\end{center}
\end{figure}

\subsection{Layer-averaging of the Euler solution}
\label{subsec:depth-averaging}

In this section we take the vertical average of the Euler system and
study the necessary closure relations for this system.

Let us denote $\avg{f}_\alpha$ the integral along the vertical axis in
the layer $\alpha$ of the quantity $f=f(z)$
i.e.
\begin{equation}
\avg{f}_\alpha(x,t) = \int_\r f(x,z,t)
\1_{z \in L_\alpha(x,t)} dz,
\label{eq:mean1}
\end{equation}
where $\1_{z \in L_\alpha(x,t)}(z)$ is the characteristic
function of the layer $\alpha$.

The goal is to propose a new derivation of the so-called {\it multilayer model
with mass exchanges} \cite{JSM_M2AN,JSM_JCP} using the
entropy-based moment closures proposed by Levermore in
\cite{levermore_closure} for kinetic equations. This method has
already been successfully used by some of the authors in~\cite{JSM_nhyd}.

Taking into account the kinematic boundary
conditions~\eqref{eq:free_surf} and~\eqref{eq:bottom}, the layer-averaged form of the Euler system~\eqref{eq:euler_2d1}--\eqref{eq:euler_2d3}
writes
\begin{align}
  &\frac{\partial}{\partial t} \avg{\varphi}_\alpha +
  \frac{\partial}{\partial x} \avg{\varphi u}_\alpha = G_{\alpha+1/2}
  - G_{\alpha-1/2} , \label{eq:av1}\\
  &\frac{\partial}{\partial t} \avg{\varphi u}_\alpha +
  \frac{\partial}{\partial x} \avg{\varphi u^2}_\alpha + \avg{
    \frac{\partial p}{\partial x} }_\alpha = u_{\alpha+1/2} G_{\alpha+1/2}
  - u_{\alpha-1/2} G_{\alpha-1/2} , \label{eq:av2}\\
  &\avg{
    \frac{\partial p}{\partial z}}_\alpha = -\avg{\varphi
    g}_\alpha,\label{eq:av3}\\
&\frac{\partial}{\partial t} \avg{\varphi z}_\alpha + \frac{\partial}{\partial
  x} \avg{\varphi z u}_\alpha = \avg{\varphi w}_\alpha + z_{\alpha+1/2}
G_{\alpha+1/2} - z_{\alpha-1/2} G_{\alpha-1/2},\label{eq:av4}
\end{align}
for $\alpha\in\{1,\ldots,N\}$ and where $p$ is defined by Eq.~\eqref{eq:pres}. The quantity $G_{\alpha+1/2}$ is
defined by
\begin{equation}
G_{\alpha+1/2} = \varphi_{\alpha+1/2} \left(
\frac{\partial  z_{\alpha+1/2}}{\partial t}
+ u_{\alpha+1/2} \frac{\partial  z_{\alpha+1/2}}{\partial x} 
- w_{\alpha+1/2}\right),\label{eq:Qalpha}
\end{equation}
and corresponds to the mass flux leaving/entering the layer $\alpha$ through
the interface $z_{\alpha+1/2}$. The value of $\varphi_{\alpha+1/2}$ is equal to 1 for every $\alpha$. Notice that the kinematic boundary conditions 
(\ref{eq:free_surf}) and (\ref{eq:bottom}) can be written
\begin{equation}
 G_{1/2} = 0,\quad G_{N+1/2}=0.
\label{eq:Qlim}
\end{equation}
These equations just express that there is no loss/supply of mass
through the bottom and the free surface. Taking into account the
condition~\eqref{eq:Qlim}, the sum for $j=1,\ldots\alpha$ of the
relations~\eqref{eq:av1} gives
\begin{equation}
G_{\alpha+1/2} = \frac{\partial }{\partial t }\sum_{j=1}^\alpha \avg{\varphi}_j
+  \frac{\partial }{\partial x }\sum_{j=1}^\alpha \avg{\varphi u}_j.
\label{eq:Qalphater}
\end{equation}
The quantities
\begin{equation}
u_{\alpha+1/2}  = u(x,z_{\alpha+1/2},t),
\label{eq:u_int}
\end{equation}
corresponding to the velocities values on the interfaces will be
defined later. Notice that when using the
expression~\eqref{eq:Qalphater}, the velocities $w_{\alpha+1/2}$ no
more appear in Eqs.~\eqref{eq:av1}-\eqref{eq:av4} and thus need not be defined.

Equation~\eqref{eq:av4} is a rewriting of
$$\avg{ \int_{z_{\alpha-1/2}}^z \left(\frac{\partial {\varphi}}{\partial {t}} + \frac{\partial {\varphi
    u}}{\partial {x}} + \frac{\partial {\varphi w}}{\partial
  {z}}\right) dz}_\alpha = \avg{z \left( \frac{\partial {\varphi}}{\partial {t}} + \frac{\partial {\varphi
    u}}{\partial {x}} + \frac{\partial {\varphi w}}{\partial {z}}\right)}_\alpha
 =0,$$
using again the kinematic boundary conditions. Notice also that
because of the hydrostatic assumption, Eq.~\eqref{eq:av4} is not a
kinematic constraint over the velocity field but the definition of the
vertical velocity $\avg{\varphi w}_\alpha$. The form of
Eq.~\eqref{eq:av4} is useful to derive energy balances but other
equivalent writings can be used, see paragraph~\ref{subsec:cal_w}.

Simple manipulations allow to obtain the
system~\eqref{eq:av1}-\eqref{eq:Qalpha} from the Euler
system~\eqref{eq:euler_2d1}-\eqref{eq:euler_2d3} with~\eqref{eq:free_surf} and~\eqref{eq:bottom} e.g. for Eq.~\eqref{eq:av1},
starting from~\eqref{eq:euler_2d1} we write
$$\avg{\frac{\partial \varphi}{\partial t} + \frac{\partial \varphi u}{\partial
    x} + \frac{\partial \varphi w}{\partial z}}_\alpha = 0,$$
and using the Leibniz rule to permute the derivative and the integral 
directly gives~\eqref{eq:av1}. Likewise, the Leibniz rule written for the pressure $p$ gives
$$\avg{\frac{\partial p}{\partial x} }_\alpha =
\int_{z_{\alpha-1/2}}^{z_{\alpha+1/2}} \frac{\partial p}{\partial x}
dz = \frac{\partial }{\partial x}\avg{p}_\alpha -
p_{\alpha+1/2}\frac{\partial z_{\alpha+1/2}}{\partial x} + p_{\alpha-1/2}\frac{\partial z_{\alpha-1/2}}{\partial x},$$
and from~\eqref{eq:av3},\eqref{eq:ps}, we get
\begin{equation}
p_{\alpha+1/2} = p(x,z_{\alpha+1/2},t) = \sum_{j=\alpha+1}^N \avg{\varphi g}_j.
\label{eq:p_mid}
\end{equation}
From Eq.~\eqref{eq:pres}, we also have
$$\avg{\frac{\partial p}{\partial x} }_\alpha =
\int_{z_{\alpha-1/2}}^{z_{\alpha+1/2}} g \frac{\partial
}{\partial x} \bigl(\varphi (\eta - z)\bigr)
dz = \frac{\partial}{\partial x}\Bigl( \frac{g}{2}\avg{\varphi}_\alpha
H\Bigr) + g \avg{\varphi}_\alpha\frac{\partial z_b}{\partial x}.$$
Relation~\eqref{eq:pres} also leads to
\begin{equation*}
p =p_{\alpha+1/2} + g \varphi(z_{\alpha+1/2}-z)=p_{\alpha-1/2} + g \varphi(z_{\alpha-1/2}-z),
\label{eq:pres1}
\end{equation*}
and hence
\begin{equation}
\avg{p}_\alpha = \avg{\varphi}_\alpha\frac{p_{\alpha+1/2} +
  p_{\alpha-1/2}}{2}=\avg{\varphi}_\alpha p_{\alpha+1/2} + \frac{g}{2}\avg{\varphi}_\alpha^2.
\label{eq:p_moy}
\end{equation}
Therefore, the system~\eqref{eq:av1}-\eqref{eq:Qalpha} can be rewritten under the form
\begin{align}
  &\frac{\partial}{\partial t} \avg{\varphi}_\alpha +
  \frac{\partial}{\partial x} \avg{\varphi u}_\alpha = G_{\alpha+1/2}
  - G_{\alpha-1/2} , \label{eq:av_h}\\
  &\frac{\partial}{\partial t} \avg{\varphi u}_\alpha +
  \frac{\partial}{\partial x} \left( \avg{\varphi u^2}_\alpha +
    \avg{p}_\alpha \right) = u_{\alpha+1/2} G_{\alpha+1/2}
  - u_{\alpha-1/2} G_{\alpha-1/2} \nonumber\\
& \hspace*{6cm} + p_{\alpha+1/2}\frac{\partial z_{\alpha+1/2}}{\partial x} - p_{\alpha-1/2}\frac{\partial z_{\alpha-1/2}}{\partial x}, \label{eq:av_u}\\
&\frac{\partial}{\partial t} \avg{\varphi z}_\alpha + \frac{\partial}{\partial
  x} \avg{\varphi z u}_\alpha = \avg{\varphi w}_\alpha + z_{\alpha+1/2}
G_{\alpha+1/2} - z_{\alpha-1/2} G_{\alpha-1/2},\label{eq:av_w}
\end{align}
with~\eqref{eq:p_mid},\eqref{eq:p_moy} and
completed with relations~\eqref{eq:Qalphater}.

Considering smooth solutions, multiplying~\eqref{eq:euler_2d2} by $u$ and integrating it over the
layer $\alpha$ gives, after simple manipulations, the energy balance
\begin{multline}
\frac{\partial}{\partial t} \avg{ E}_\alpha +
  \frac{\partial}{\partial x} \avg{ u(E+p)}_\alpha =  
\left( \frac{u_{\alpha+1/2}^2}{2} + p_{\alpha+1/2} + g z_{\alpha+1/2}\right) G_{\alpha+1/2}\\
-\left( \frac{u_{\alpha-1/2}^2}{2} + p_{\alpha-1/2} + g
  z_{\alpha-1/2}\right) G_{\alpha-1/2}  - p_{\alpha+1/2}\frac{\partial z_{\alpha+1/2}}{\partial t} +
p_{\alpha-1/2}\frac{\partial z_{\alpha-1/2}}{\partial t},
\label{eq:energy_av}
\end{multline}
where $E=E(z;u)$ is defined by~\eqref{eq:energy_exp}. The sum for
$\alpha=1,\ldots,N$ of the relations~\eqref{eq:energy_av} gives
$$\frac{\partial}{\partial t} \sum_{\alpha=1}^N \avg{ E}_\alpha+
  \frac{\partial}{\partial x} \sum_{\alpha=1}^N \avg{ u(E+p)}_\alpha = 0.$$

Therefore the
system~\eqref{eq:av_h}-\eqref{eq:av_w} completed~\eqref{eq:Qalphater},~\eqref{eq:p_mid} and~\eqref{eq:p_moy} has three equations with three
unknowns, namely $\avg{\varphi}_\alpha$, $\avg{\varphi u}_\alpha$ and~$\avg{\varphi w}_\alpha$
and closure relations are needed to define $\avg{\varphi u^2}_\alpha$,
$\avg{\varphi z u}_\alpha$ and $u(x,z_{\alpha+1/2},t)$.

\subsection{Closure relations}

If $u^\prime_\alpha$ is defined as the
deviation of $u$ with respect to its layer-average over the
layer $\alpha$, then it comes for $z\in L_\alpha$
\begin{equation}
  \label{eq:uaverage}
 \varphi u = \frac{\avg{\varphi u}_\alpha}{\avg{\varphi}_\alpha} + \varphi u^\prime_\alpha,
\end{equation}
with $\avg{\varphi u^\prime_\alpha} = 0$. Following the moment closure proposed by Levermore~\cite{levermore_closure}, we
study the minimization problem
\begin{equation}
  \label{eq:minSW}
\min_{u^\prime_\alpha} \avg{ \{ \varphi E(z;u)  \}}_\alpha.
\end{equation}
The energy $E(z;u) $ being quadratic with respect
to $u$ we notice that
\begin{eqnarray}
    \avg{\varphi u^2}_\alpha & = & \frac{\avg{\varphi u}^2_\alpha}{\avg{\varphi}_\alpha} +
    \frac{2 \avg{\varphi u u^\prime}_\alpha}{\avg{\varphi}_\alpha} +
    \avg{\varphi (u^\prime_\alpha)^2}_\alpha \nonumber\\
    &= & \frac{\avg{\varphi  u}^2_\alpha}{\avg{\varphi}_\alpha} +
    \avg{\varphi  (u^\prime_\alpha)^2}_\alpha\nonumber\\
    &\geq & \frac{\avg{\varphi  u}^2_\alpha}{\avg{\varphi}_\alpha}.
\label{eq:uaverage2}
\end{eqnarray}
Equation~\eqref{eq:uaverage2} means that the solution
of the minimization problem~\eqref{eq:minSW} is given by
\begin{equation}
  \label{eq:minSW1}
\avg{\varphi  E\left(z;\frac{\avg{\varphi  u}_\alpha}{\avg{\varphi }_\alpha}\right)}_\alpha = \min_{u^\prime_\alpha} \avg{ \{ \varphi  E(z;u)  \}}_\alpha.
\end{equation}
and
\begin{equation}
\avg{\varphi  E\left(z;\frac{\avg{\varphi  u}_\alpha}{\avg{\varphi
      }_\alpha}\right)}_\alpha = \frac{\avg{\varphi
    u}^2_\alpha}{2\avg{\varphi}_\alpha} + g\avg{\varphi z}_\alpha.
\label{eq:minSW11}
\end{equation}
Since the only choice leading to an equality in
relation~\eqref{eq:uaverage2} corresponds to
\begin{equation}
\varphi u = \frac{\avg{\varphi
    u}_\alpha}{\avg{\varphi}_\alpha},\quad\mbox{for } z\in L_\alpha,
\label{eq:closures}
\end{equation}
this allows to precise the closure relation associated to a minimal
energy, namely
\begin{align}
& \avg{\varphi u^2}_\alpha = \frac{\avg{\varphi u}_\alpha^2}{\avg{\varphi}_\alpha},\label{eq:close1}\\
& \avg{\varphi z u}_\alpha = \avg{\varphi z}_\alpha\frac{\avg{\varphi u}_\alpha}{\avg{\varphi}_\alpha}. \label{eq:close3}
\end{align}
It remains to define the quantities $u_{\alpha+1/2}$. We adopt the definition
\begin{equation}
u_{\alpha+1/2} =
\left\{\begin{array}{ll}
\frac{\avg{\varphi u}_\alpha}{\avg{\varphi}_\alpha} & \mbox{if } \;G_{\alpha+1/2} \leq 0\\
\frac{\avg{\varphi u}_{\alpha+1}}{\avg{\varphi}_{\alpha+1}} & \mbox{if } \;G_{\alpha+1/2} > 0
\end{array}\right.
\label{eq:upwind_uT}
\end{equation}
corresponding to an upwind definition, depending on the mass exchange sign between
the layers $\alpha$ and $\alpha+1$. This choice is justified by the
form of energy balance in the following proposition.
\begin{proposition}
The solutions of the
Euler system~\eqref{eq:euler_2d1}-\eqref{eq:euler_2d3} with~\eqref{eq:free_surf},\eqref{eq:bottom} satisfying the
closure relations~\eqref{eq:close1}-\eqref{eq:upwind_uT} are also
solutions of the system
\begin{align}
  &\frac{\partial}{\partial t} \avg{\varphi}_\alpha +
  \frac{\partial}{\partial x} \avg{\varphi u}_\alpha = G_{\alpha+1/2}
  - G_{\alpha-1/2} , \label{eq:av_h1}\\
  &\frac{\partial}{\partial t} \avg{\varphi u}_\alpha +
  \frac{\partial}{\partial x} \left( \frac{\avg{\varphi u}_\alpha^2}{\avg{\varphi}_\alpha} +
    \avg{ p}_\alpha \right) = u_{\alpha+1/2} G_{\alpha+1/2}
  - u_{\alpha-1/2} G_{\alpha-1/2} \nonumber\\
& \hspace*{6cm} + p_{\alpha+1/2}\frac{\partial z_{\alpha+1/2}}{\partial x} - p_{\alpha-1/2}\frac{\partial z_{\alpha-1/2}}{\partial x}, \label{eq:av_u1}\\
&\frac{\partial}{\partial t} \avg{\varphi z}_\alpha + \frac{\partial}{\partial
  x} \left(\avg{\varphi z}_\alpha\frac{\avg{\varphi u}_\alpha}{\avg{\varphi}_\alpha}\right) = \avg{\varphi w}_\alpha + z_{\alpha+1/2}
G_{\alpha+1/2} - z_{\alpha-1/2} G_{\alpha-1/2},\label{eq:av_w1}
\end{align}
completed with relation~\eqref{eq:Qalphater}. The quantities $\avg{p}_\alpha$ and
$p_{\alpha+1/2}$ are defined by~\eqref{eq:p_mid} and~\eqref{eq:p_moy}.

This system is a layer-averaged approximation of the
Euler system and
admits~-- for smooth solutions~-- an energy equality
under the form
\begin{multline}
  \frac{\partial}{\partial t} \sum_{\alpha=1}^N
  \avg{E\left(z;\frac{\avg{ \varphi u}_\alpha}{\avg{\varphi}_\alpha}\right)}_\alpha +
  \frac{\partial}{\partial x} \sum_{\alpha=1}^N
  \avg{\frac{\avg{\varphi u}_\alpha}{\avg{\varphi}_\alpha} \left(
    E\left(z;\frac{\avg{\varphi u}_\alpha}{\avg{\varphi}_\alpha}\right) + \avg{p}_\alpha
  \right)}_\alpha = \\
 - \frac{1}{2} \sum_{\alpha=1}^N \left( \frac{\avg{\varphi u}_{\alpha+1}}{\avg{\varphi}_{\alpha+1}} -
  \frac{\avg{\varphi u}_\alpha}{\avg{\varphi}_\alpha} \right)^2 |G_{\alpha+1/2}|.
\label{eq:energy_av_fin}
\end{multline}
\label{prop:prop_ad}
\end{proposition}

\begin{remark}
Instead of~\eqref{eq:upwind_uT}, the definition
\begin{equation}
u_{\alpha+1/2} = \frac{1}{2}\left( \frac{\avg{\varphi
      u}_{\alpha+1}}{\avg{\varphi}_{\alpha+1}} + \frac{\avg{\varphi u}_{\alpha}}{\avg{\varphi}_{\alpha}} \right),
\label{eq:upwind_uT_c}
\end{equation}
is also possible and gives a vanishing right hand side in~\eqref{eq:energy_av_fin}. But such a choice
does not allow to obtain an energy balance in the variable density
case and does not give a maximum principle, at the discrete level,
see~\cite{JSM_JCP}. Simple calculations show that any other choice
that~\eqref{eq:upwind_uT} or~\eqref{eq:upwind_uT_c} leads to a non
negative r.h.s. in~\eqref{eq:energy_av_fin}, see Eq.~\eqref{eq:energy_av_fin_mc} in the
proof of prop.~\ref{prop:prop_ad}.
\label{rem:u_int}
\end{remark}

\begin{remark}
It is important to notice that whereas the solution $H,u,w,p$ of the
Euler system~\eqref{eq:euler_2d1}-\eqref{eq:ps},\eqref{eq:free_surf},\eqref{eq:bottom} also satisfies the
system~\eqref{eq:av_h}-\eqref{eq:av_w},
only the solutions $H,u,w,p$ of the
Euler system~\eqref{eq:euler_2d1}-\eqref{eq:ps},\eqref{eq:free_surf},\eqref{eq:bottom} satisfying the
closure relations~\eqref{eq:close1}-\eqref{eq:close3},\eqref{eq:upwind_uT} are also
solutions of the system~\eqref{eq:av_h1}-\eqref{eq:energy_av_fin}. On the
contrary, any solutions $\avg{\varphi}_\alpha$, $\avg{\varphi u}_\alpha$,
$\avg{\varphi w}_\alpha$ and~$\avg{p}_\alpha$ 
of~\eqref{eq:av_h1}-\eqref{eq:av_w1} with~\eqref{eq:upwind_uT} are also
solutions of~\eqref{eq:av_h}-\eqref{eq:energy_av}.
\end{remark}

\begin{proof}[Proof of prop.~\ref{prop:prop_ad}]
Only the manipulations allowing to obtain~\eqref{eq:energy_av_fin}
have to be detailed. For that purpose, we multiply~\eqref{eq:av_u1} by
$\frac{\avg{\varphi u}_\alpha}{\avg{\varphi}_\alpha}$ giving
\begin{multline*}
\left( \frac{\partial}{\partial t } \frac{\avg{\varphi u}_\alpha}{\avg{\varphi}_\alpha} +
\frac{\partial }{\partial x }\left( \frac{\avg{\varphi u}_\alpha^2}{\avg{\varphi}_\alpha} +
  \avg{p}_\alpha\right) \right) \frac{\avg{\varphi u}_\alpha}{\avg{\varphi}_\alpha} = \biggl(
u_{\alpha+1/2}G_{\alpha+1/2} -
u_{\alpha-1/2}G_{\alpha-1/2} \\
+\frac{\partial z_{\alpha+1/2}}{\partial x} p_{\alpha+1/2} -
\frac{\partial z_{\alpha-1/2}}{\partial x} p_{\alpha-1/2} \biggr)
\frac{\avg{\varphi u}_\alpha}{\avg{\varphi}_\alpha},
\end{multline*}
and we rewrite each of the obtained
terms.

Considering first the left hand side of the preceding equation
excluding the pressure terms, we denote
$$I_{u,\alpha}  = \left( \frac{\partial}{\partial t } \frac{\avg{\varphi u}_\alpha}{\avg{\varphi}_\alpha} +
\frac{\partial }{\partial x }\left( \frac{\avg{\varphi u}_\alpha^2}{\avg{\varphi}_\alpha} \right) \right) \frac{\avg{\varphi u}_\alpha}{\avg{\varphi}_\alpha},$$
and using~\eqref{eq:av1} we have
$$I_{u,\alpha}  = \frac{\partial }{\partial t } \left( \frac{\avg{\varphi u}_\alpha^2}{2\avg{\varphi}_\alpha} \right) + \frac{\partial }{\partial x}
\left(  \frac{\avg{\varphi u}_\alpha}{\avg{\varphi }_\alpha} \frac{\avg{\varphi u}_\alpha^2}{2\avg{\varphi }_\alpha} \right)
+ \frac{\avg{\varphi u}_\alpha^2}{2\avg{\varphi }_\alpha^2} \left( G_{\alpha+1/2} - G_{\alpha-1/2}\right).$$
Now we consider the contribution of the pressure terms over the energy
balance i.e.
$$I_{p,\alpha} = \left( \frac{\partial \avg{p}_\alpha}{\partial x } - p_{\alpha+1/2}\frac{\partial z_{\alpha+1/2}}{\partial x} +
p_{\alpha-1/2} \frac{\partial z_{\alpha-1/2}}{\partial x} \right) \frac{\avg{\varphi u}_\alpha}{\avg{\varphi}_\alpha}.$$
Using~\eqref{eq:p_moy} we get the equality
$$p_{\alpha+1/2}\frac{\partial z_{\alpha+1/2}}{\partial x} -
p_{\alpha-1/2} \frac{\partial z_{\alpha-1/2}}{\partial x} =
\frac{\avg{p}_\alpha}{\avg{\varphi}_\alpha}
\frac{\partial h_\alpha}{\partial x} - \avg{g\varphi}_\alpha\frac{\partial
z_\alpha}{\partial x},$$
holds, it comes
\begin{eqnarray*}
I_{p,\alpha} & = & \frac{\partial }{\partial x}\left( \avg{p}_\alpha
  \frac{\avg{\varphi u}_\alpha}{\avg{\varphi}_\alpha} \right) -
\avg{p}_\alpha\frac{\partial }{\partial x} \left( \frac{\avg{\varphi u}_\alpha}{\avg{\varphi}_\alpha}\right)-
\frac{\avg{p}_\alpha}{\avg{\varphi}_\alpha}\frac{\avg{\varphi u}_\alpha}{\avg{\varphi}_\alpha} \frac{\partial \avg{\varphi}_\alpha}{\partial x} +
g \avg{\varphi u}_\alpha\frac{\partial z_\alpha
}{\partial x}\\
& = & \frac{\partial}{\partial
  x}\left( \avg{p}_\alpha
  \frac{\avg{\varphi u}_\alpha}{\avg{\varphi}_\alpha} \right) - \frac{\avg{p}_\alpha}{\avg{\varphi}_\alpha}\frac{\partial \avg{\varphi u}_\alpha}{\partial x} +
\frac{\partial
}{\partial x} \left( g h_\alpha z_{\alpha} \frac{\avg{\varphi u}_\alpha}{\avg{\varphi}_\alpha} \right)
-z_{\alpha}\frac{\partial
}{\partial x} \left( g \avg{\varphi u}_\alpha \right)\\
& = & \frac{\partial }{\partial
  x} \left( \avg{p}_\alpha \frac{\avg{\varphi u}_\alpha}{\avg{\varphi}_\alpha} \right)- \frac{\avg{p}_\alpha}{\avg{\varphi}_\alpha}\frac{\partial \avg{\varphi u}_\alpha}{\partial x} +
\frac{\partial
}{\partial x} \left( g h_\alpha z_{\alpha} \frac{\avg{\varphi u}_\alpha}{\avg{\varphi}_\alpha}\right)
+z_{\alpha}\frac{\partial
}{\partial t} \left(  \avg{g\varphi} _\alpha  \right)\\
& & -g z_{\alpha} \left(G_{\alpha+1/2} -
  G_{\alpha-1/2}\right)\\
& = & \frac{\partial }{\partial
  x} \left( \avg{p}_\alpha \frac{\avg{\varphi u}_\alpha}{\avg{\varphi}_\alpha}\right) -
\frac{\avg{p}_\alpha}{\avg{\varphi}_\alpha}\frac{\partial \avg{\varphi u}_\alpha}{\partial x} +
\frac{\partial
}{\partial x} \left( g h_\alpha z_{\alpha} \frac{\avg{\varphi u}_\alpha}{\avg{\varphi}_\alpha}\right)
+\frac{\partial
}{\partial t} \left( g h_\alpha z_{\alpha}  \right)\\
& & - g h_\alpha\frac{\partial z_\alpha
}{\partial t} -g z_{\alpha} \left(G_{\alpha+1/2} - G_{\alpha-1/2}\right).
\end{eqnarray*}
Let us rewrite $I_{p,\alpha} $ under the form
\begin{eqnarray*}
I_{p,\alpha} 
& = & \frac{\partial }{\partial
  x} \left( \avg{p}_\alpha \frac{\avg{\varphi u}_\alpha}{\avg{\varphi}_\alpha} \right)+\frac{\partial
}{\partial t} \left( g h_\alpha z_{\alpha}  \right) +
g\frac{\partial
}{\partial x} \left( h_\alpha z_{\alpha} \frac{\avg{u}_\alpha}{h_\alpha}\right)\\
& & 
-g\left( z_{\alpha+1/2} G_{\alpha+1/2} -
 z_{\alpha-1/2} G_{\alpha-1/2}\right) + J_{p,\alpha},
\end{eqnarray*}
with
$$ J_{p,\alpha} = - \frac{\avg{p}_\alpha}{\avg{\varphi}_\alpha}\frac{\partial \avg{\varphi u}_\alpha}{\partial
  x} - g h_\alpha\frac{\partial z_{\alpha}
}{\partial t} +
g\frac{h_\alpha}{2} \left(G_{\alpha+1/2} + G_{\alpha-1/2}\right).$$
Since we have
$$
\frac{\avg{p}_\alpha}{\avg{\varphi}_\alpha}\frac{\partial \avg{\varphi u}_\alpha}{\partial x} 
= \frac{\avg{\varphi}_\alpha}{\avg{\varphi}_\alpha}\left( G_{\alpha+1/2} -
  G_{\alpha-1/2} - \frac{\partial
    h_\alpha}{\partial t} \right),
$$
we obtain
\begin{eqnarray*}
J_{p,\alpha} 
& = & p_{\alpha+1/2}\frac{\partial z_{\alpha+1/2}}{\partial
  t} - p_{\alpha-1/2}\frac{\partial z_{\alpha-1/2}}{\partial
  t} - p_{\alpha+1/2} G_{\alpha+1/2} +
p_{\alpha-1/2} G_{\alpha-1/2}.
\end{eqnarray*}
Then summing $I_{u,\alpha}$ and $I_{p,\alpha}$ gives
\begin{multline*}
  \frac{\partial}{\partial t} \avg{E\left(z;\frac{\avg{\varphi u}_\alpha}{\avg{\varphi}_\alpha}\right)}_\alpha +
  \frac{\partial}{\partial x} 
  \avg{\frac{\avg{\varphi u}_\alpha}{\avg{\varphi}_\alpha} \left(
    E\left(z;\frac{\avg{\varphi u}_\alpha}{\avg{\varphi}_\alpha}\right) + \avg{p}_\alpha
  \right)}_\alpha = \\
 \left( u_{\alpha+1/2} \frac{\avg{\varphi u}_\alpha}{\avg{\varphi}_\alpha} -
  \frac{1}{2}\left(\frac{\avg{\varphi u}_\alpha}{\avg{\varphi}_\alpha}\right)^2\right)
G_{\alpha+1/2} -  \left( u_{\alpha-1/2} \frac{\avg{\varphi u}_\alpha}{\avg{\varphi}_\alpha} -
  \frac{1}{2}\left(\frac{\avg{\varphi u}_\alpha}{\avg{\varphi}_\alpha}\right)^2\right) G_{\alpha-1/2}.
\end{multline*}
Finally, the sum of the preceding relations for $\alpha=1,\ldots,N$
\begin{multline}
  \frac{\partial}{\partial t} \sum_{\alpha=1}^N \avg{E\left(z;\frac{\avg{\varphi u}_\alpha}{\avg{\varphi}_\alpha}\right)}_\alpha +
  \frac{\partial}{\partial x} \sum_{\alpha=1}^N
  \avg{\frac{\avg{\varphi u}_\alpha}{\avg{\varphi}_\alpha} \left(
    E\left(z;\frac{\avg{\varphi u}_\alpha}{\avg{\varphi}_\alpha}\right) + \avg{p}_\alpha
  \right)}_\alpha = \\
 \sum_{\alpha=1}^N \left( u_{\alpha+1/2} \left(\frac{\avg{\varphi u}_\alpha}{\avg{\varphi}_\alpha} -\frac{\avg{\varphi u}_{\alpha+1}}{\avg{\varphi}_{\alpha+1}}\right) -
  \frac{1}{2}\left(\frac{\avg{\varphi u}_\alpha}{\avg{\varphi}_\alpha}\right)^2 +
  \frac{1}{2}\left(\frac{\avg{\varphi u}_{\alpha+1}}{\avg{\varphi}_{\alpha+1}}\right)^2\right)
G_{\alpha+1/2},
\label{eq:energy_av_fin_mc}
\end{multline}
and the definition~\eqref{eq:upwind_uT} gives relation~\eqref{eq:energy_av_fin} that
completes the proof. Notice that any other choice than~\eqref{eq:upwind_uT}
or~\eqref{eq:upwind_uT_c} leads to a non negative
r.h.s. in~\eqref{eq:energy_av_fin_mc}, see remark~\ref{rem:u_int}.
\end{proof}

\section{The proposed layer-averaged Euler system}
\label{sec:euler_mc}

\subsection{Formulation}

The closure relations~\eqref{eq:close1}-\eqref{eq:close3} motivate the definition of
piecewise constant approximation of the variables $u$ and $w$.

Let us consider the space
$\Bbb{P}_{0,H}^{N,t}$ of piecewise constant functions defined by
$$\Bbb{P}_{0,H}^{N,t} = \left\{ \1_{z \in L_\alpha(x,t)}(z),\quad \alpha\in\{1,\ldots,N\}\right\}.$$ 
Using this formalism, the
projection of $u$ and $w$ on $\Bbb{P}_{0,H}^{N,t}$ is a
piecewise constant function defined by
\begin{equation}
X^{N}(x,z,\{z_\alpha\},t)  = \sum_{\alpha=1}^N
\1_{]z_{\alpha-1/2},z_{\alpha+1/2}[}(z)X_\alpha(x,t),
\label{eq:ulayer}
\end{equation}
for $X \in (u,w)$. 
In the following, we no more handle variables corresponding to vertical means of the
solution of the Euler
equations~\eqref{eq:euler_2d1}-\eqref{eq:euler_2d3} and we adopt
notations inherited from~\eqref{eq:ulayer}.

By analogy with~\eqref{eq:av_h1}-\eqref{eq:av_w1} we consider the following model
\begin{eqnarray}
&&\sum_{\alpha=1}^N \frac{\partial h_\alpha }{\partial t }   + 
\sum_{\alpha=1}^N \frac{\partial (h_\alpha u_\alpha)}{\partial x } =0,
\label{eq:H}\\
&&\frac{\partial  h_\alpha u_\alpha}{\partial t } +
\frac{\partial }{\partial x }\left(h_\alpha u^2_\alpha +
  h_\alpha p_\alpha\right)=
u_{\alpha+1/2}G_{\alpha+1/2} -
u_{\alpha-1/2}G_{\alpha-1/2}\nonumber\\
&& \hspace*{5.7cm}
+\frac{\partial z_{\alpha+1/2}}{\partial x} p_{\alpha+1/2} - \frac{\partial z_{\alpha-1/2}}{\partial x} p_{\alpha-1/2},
\label{eq:eq4}\\
& & \frac{\partial}{\partial t} \left(\frac{z_{\alpha+1/2}^2 - z_{\alpha-1/2}^2}{2}\right) +
\frac{\partial}{\partial x} \left( \frac{z_{\alpha+1/2}^2 -
    z_{\alpha-1/2}^2}{2}u_\alpha\right) = h_\alpha
w_\alpha\nonumber\\
& & \hspace*{6cm} + z_{\alpha+1/2}G_{\alpha+1/2} - z_{\alpha-1/2}
G_{\alpha-1/2},\label{eq:eq6}
\end{eqnarray}
by analogy with~\eqref{eq:Qalphater}
\begin{equation}
G_{\alpha+1/2} =  \frac{\partial }{\partial t }\sum_{j=1}^\alpha h_j
+  \frac{\partial }{\partial x }\sum_{j=1}^\alpha \left( h_j u_j \right),
\label{eq:Qalphabis}
\end{equation}
and we have $p_\alpha$, $p_{\alpha+1/2}$ 
given by
\begin{equation}
p_{\alpha} = g \left(\frac{h_\alpha}{2} + \sum_{j=\alpha+1}^N
h_j \right) \qquad \mbox{\rm and}\qquad p_{\alpha+1/2} = 
g\sum_{j=\alpha+1}^N h_j.
\label{eq:palpha}
\end{equation}
The definition of $u_{\alpha+1/2}$ is equivalent to~\eqref{eq:upwind_uT}
i.e.
\begin{equation*}
u_{\alpha+1/2} =
\left\{\begin{array}{ll}
u_\alpha & \mbox{if } \;G_{\alpha+1/2} \leq 0\\
u_{\alpha+1} & \mbox{if } \;G_{\alpha+1/2} > 0
\end{array}\right.
\label{eq:upwind_uT_n}
\end{equation*}

The smooth solutions of~\eqref{eq:H}-\eqref{eq:eq6} satisfy the energy
balance
\begin{eqnarray}
& & \hspace*{-0.5cm}\frac{\partial }{\partial t} E_{\alpha} + \frac{\partial}{\partial x}
\left( u_\alpha\left(E_{\alpha} + h_\alpha p_\alpha \right)
\right) =
\left( u_{\alpha+1/2}u_{\alpha} -\frac{u_{\alpha}^2}{2} + p_{\alpha+1/2} + g z_{\alpha+1/2}\right) G_{\alpha+1/2}\nonumber \\
& & \qquad\qquad
-\left( u_{\alpha-1/2}u_\alpha -\frac{u_{\alpha}^2}{2} + p_{\alpha-1/2} + g
  z_{\alpha-1/2}\right) G_{\alpha-1/2} \nonumber \\
& & \qquad\qquad -p_{\alpha+1/2}\frac{\partial z_{\alpha+1/2}}{\partial t} +
p_{\alpha-1/2}\frac{\partial z_{\alpha-1/2}}{\partial t},
\label{eq:energy_mc}
\end{eqnarray}
with
$$E_{\alpha}=\frac{h_\alpha
  u_\alpha^2}{2}+\frac{g}{2}(z_{\alpha+1/2}^2 -
  z_{\alpha-1/2}^2) = h_\alpha \biggl( \frac{u_\alpha^2}{2} + g z_\alpha \biggr).$$
Adding the preceding relations for $\alpha=1,\ldots,N$, we obtain the global equality
\begin{eqnarray}
\frac{\partial }{\partial t} \left( \sum_{\alpha=1}^N E_{\alpha}\right) + \frac{\partial}{\partial x}
\left(\sum_{\alpha=1}^N u_\alpha\left(E_{\alpha}+
    h_\alpha p_\alpha \right) \right) = -\sum_{\alpha=1}^N \frac{1}{2}(u_{\alpha+1/2}
- u_\alpha)^2 |G_{\alpha+1/2}|.
\label{eq:energy_glo}
\end{eqnarray}
Using~\eqref{eq:palpha}, the pressure terms in~\eqref{eq:eq4} can be
rewritten under the form
\begin{equation}
\frac{\partial}{\partial x} \left(h_\alpha p_\alpha\right) -
\frac{\partial z_{\alpha+1/2}}{\partial x} p_{\alpha+1/2} +
\frac{\partial z_{\alpha-1/2}}{\partial x} p_{\alpha-1/2} =
\frac{\partial}{\partial x} \left(\frac{g}{2}H h_\alpha\right) +
gh_\alpha\frac{\partial z_b}{\partial x}.
\label{eq:pres_interf}
\end{equation}

\subsection{The vertical velocity}
\label{subsec:cal_w}

The equation~\eqref{eq:eq6} is a definition of the vertical velocity $w^N $ given by \eqref{eq:ulayer}. The quantities $w_\alpha$ are
not unknowns of the problem but only output variables. Indeed, once $H$
and $u^N$ have been calculated
solving~\eqref{eq:H},\eqref{eq:eq4} with~\eqref{eq:Qalphabis}, the
vertical velocities $w_\alpha$ can be determined using~\eqref{eq:eq6}.

Using simple manipulations, Eq.~\eqref{eq:eq6} can be rewritten under
several forms. In particular, the following proposition holds

\begin{proposition}
Let us introduce $\hat{w} = \hat{w}(x,z,t)$ defined by
\begin{equation}\frac{\partial u^N}{\partial x} + \frac{\partial
  \hat{w}}{\partial z} =0,
  \label{eq:conti_w}
\end{equation}
The quantity $\hat{w}$ is affine in $z$ and discontinuous at each interface
$z_{\alpha+1/2}$,
$\hat{w}$ can be written:
\begin{equation}
\hat{w} = k_\alpha - z \frac{\partial u_\alpha}{\partial
  x},
\label{eq:w_hat}
\end{equation}
with $k_\alpha=k_\alpha (x,t)$ recursively defined by
\begin{eqnarray*}
& & k_1 = \frac{\partial (z_bu_1)}{\partial
  x},\\
& & k_{\alpha+1} = k_\alpha + \frac{\partial }{\partial
  x}\bigl( z_{\alpha+1/2}(u_{\alpha+1} - u_\alpha)\bigr).
\end{eqnarray*}
Therefore we have
\begin{equation}
\int_{z_{\alpha-1/2}}^{z_{\alpha+1/2}} \hat{w} dz = h_\alpha
w_\alpha,
\label{eq:w_hat_mean}
\end{equation}
meaning the quantities $\hat{w}$ is a natural and consistent
affine extension of the layer-averaged quantities $w_\alpha$ defined
by~\eqref{eq:eq6}. Using~\eqref{eq:w_hat_mean}, an integration along the layer
$\alpha$ of~\eqref{eq:w_hat} gives
\begin{equation}
h_\alpha w_\alpha = h_\alpha k_\alpha - \frac{z_{\alpha+1/2}^2 - z_{\alpha-1/2}^2}{2} \frac{\partial u_\alpha}{\partial x}
= h_\alpha \left(  k_\alpha - z_{\alpha}  \frac{\partial u_\alpha}{\partial x} \right).
\label{eq:mean_w}
\end{equation}
or
\begin{equation}
w_\alpha = k_\alpha - z_\alpha \frac{\partial u_\alpha}{\partial
  x} = \hat w (z_\alpha).
\end{equation}
\label{prop:w_alpha}
\end{proposition}

\begin{proof}[Proof of prop.~\ref{prop:w_alpha}]
A simple integration along $z$ of equation \eqref{eq:conti_w} using~\eqref{eq:bottom} gives
\begin{equation}
\hat w = -\frac{\partial}{\partial x} \int_{z_b}^z u^N\ dz,
\label{eq:w_int1}
\end{equation}
and therefore, for $z\in L_1$ we get
$$\hat w = -\frac{\partial}{\partial x} \int_{z_b}^z u_1\ dz =
-\frac{\partial}{\partial x} \bigl( (z-z_b) u_1\bigr),$$
i.e.
$$\hat{w} = \frac{\partial}{\partial x} (z_b u_1) - z \frac{\partial u_1}{\partial x}.$$
For $z\in L_\alpha$, relation~\eqref{eq:w_int1} gives
\begin{equation}
\hat w = -\sum_{j=1}^{\alpha-1} \frac{\partial}{\partial x} (h_j u_j) 
-\frac{\partial}{\partial x} \bigl( (z-z_{\alpha-1/2}) u_\alpha
\bigr),
\label{eq:hat_w1}
\end{equation}
and we easily obtain
$$\hat{w} = k_\alpha - z \frac{\partial u_\alpha}{\partial
  x}.$$
Now we intend to prove~\eqref{eq:w_hat_mean}.

Using the definition~\eqref{eq:zalpha}, relation~\eqref{eq:eq6} also writes
$$h_\alpha w_\alpha = \frac{\partial}{\partial x}  (z_\alpha h_\alpha
u_\alpha) - z_{\alpha+1/2} \sum_{j=1}^\alpha \frac{\partial (h_j u_j)}{\partial x}+ z_{\alpha-1/2}\sum_{j=1}^{\alpha-1} \frac{\partial (h_j u_j)}{\partial x},
$$
leading to a new expression governing
$w_\alpha$ under the form
\begin{equation}
h_\alpha w_\alpha = -\frac{h_\alpha}{2}\frac{\partial (h_\alpha u_\alpha)}{\partial x}
- h_\alpha \sum_{j=1}^{\alpha-1}
\frac{\partial (h_j u_j)}{\partial x}
+ h_\alpha u_\alpha \frac{\partial z_\alpha}{\partial x}.
\label{eq:w_alpha}
\end{equation}
And from~\eqref{eq:hat_w1}, we get
\begin{eqnarray*}
\int_{z_{\alpha-1/2}}^{z_{\alpha+1/2}}\hat{w} dz & = & -h_\alpha\sum_{j=1}^{\alpha-1} \frac{\partial}{\partial x} (h_j u_j) 
+ h_\alpha\frac{\partial}{\partial x} \bigl( z_{\alpha-1/2} u_\alpha\bigr)
- h_{\alpha} z_{\alpha}\frac{\partial
  u_\alpha}{\partial x}\\
& = & -h_\alpha\sum_{j=1}^{\alpha-1} \frac{\partial}{\partial x} (h_j u_j) 
-\frac{h_\alpha}{2} \frac{\partial
  }{\partial x}\left( h_\alpha u_\alpha\right) +h_\alpha u_\alpha\frac{\partial
 z_\alpha }{\partial x},
\end{eqnarray*}
corresponding to~\eqref{eq:w_alpha} and proving the result.
\end{proof}

Using also~\eqref{eq:pres_interf}, we are able to rewrite the system~\eqref{eq:H}-\eqref{eq:eq6} under
the form
\begin{eqnarray}
&&\sum_{\alpha=1}^N \frac{\partial h_\alpha }{\partial t }   + 
\sum_{\alpha=1}^N \frac{\partial (h_\alpha u_\alpha)}{\partial x } =0,
\label{eq:H_bis}\\
&&\frac{\partial  h_\alpha u_\alpha}{\partial t } +
\frac{\partial }{\partial x }\left(h_\alpha u^2_\alpha +
  \frac{g}{2}h_\alpha H\right)= - gh_\alpha\frac{\partial
  z_{b}}{\partial x} +
u_{\alpha+1/2}G_{\alpha+1/2} -
u_{\alpha-1/2}G_{\alpha-1/2},
\label{eq:eq4_bis}\\
&& w_\alpha = -\frac{1}{2}\frac{\partial (h_\alpha u_\alpha)}{\partial
  x} - \sum_{j=1}^{\alpha-1}
\frac{\partial (h_j u_j)}{\partial x}
+ u_\alpha \frac{\partial z_\alpha}{\partial x}.\label{eq:eq6_bis}
\end{eqnarray}

\section{The Navier-Stokes system}
\label{sec:av_NS}

Instead of considering the Euler system, we can also depart from the
Navier-Stokes equations to derive a layer-averaged model.

The model derivation is similar to what has been done in
Section~\ref{sec:av_euler} for the Euler system.

\subsection{Layer averaging of the viscous terms}
In this paragraph and the both following, the components of the Cauchy stress tensor $\Sigma$
are not specified.
It remains to find a layer-averaged formulation for
the r.h.s. of Eq.~\eqref{eq:NS_2d2_mod}, {\sl i.e.}
$$V_\alpha = \int_{z_{\alpha-1/2}}^{z_{\alpha+1/2}} \left(
  \frac{\partial \Sigma_{xx}}{\partial x} + \frac{\partial
  \Sigma_{xz}}{\partial z} + \frac{\partial^2}{\partial x^2} \int_z^\eta
\Sigma_{zx} dz_1 - \frac{\partial
    \Sigma_{zz}}{\partial x} \right) dz.$$
We have
\begin{eqnarray*}
V_\alpha & = & \frac{\partial }{\partial x}
\int_{z_{\alpha-1/2}}^{z_{\alpha+1/2}} \left( \Sigma_{xx} + \frac{\partial}{\partial x} \int_z^\eta
\Sigma_{zx} dz_1-
  \Sigma_{zz} \right) dz \\
& & + \Sigma_{xz}|_{\alpha+1/2} - \frac{\partial z_{\alpha+1/2}}{\partial x}
\left.\left(\Sigma_{xx} + \frac{\partial}{\partial x} \int_z^\eta
\Sigma_{zx} dz_1 -
  \Sigma_{zz}\right) \right|_{z_{\alpha+1/2}} \\
& & - \Sigma_{xz}|_{\alpha-1/2} + \frac{\partial z_{\alpha-1/2}}{\partial x}
\left.\left(\Sigma_{xx} + \frac{\partial}{\partial x} \int_z^\eta
\Sigma_{zx} dz_1 -
  \Sigma_{zz}\right) \right|_{z_{\alpha-1/2}}.
\end{eqnarray*}
In the expression $V_\alpha$ we have the term
\begin{align*}
\frac{\partial }{\partial x}
\int_{z_{\alpha-1/2}}^{z_{\alpha+1/2}} & \left( \frac{\partial}{\partial x} \int_z^\eta
\Sigma_{zx} dz_1\right) dz \\
= &  \frac{\partial}{\partial x}
\Biggl( \frac{\partial}{\partial
  x}\int_{z_{\alpha-1/2}}^{z_{\alpha+1/2}} \int_z^\eta \Sigma_{zx} dz_1  dz - \frac{\partial z_{\alpha+1/2}}{\partial x}\int_{z_{\alpha+1/2}}^\eta
\Sigma_{zx} dz\\
& + \frac{\partial z_{\alpha-1/2}}{\partial x}\int_{z_{\alpha-1/2}}^\eta
\Sigma_{zx} dz\Biggr)\\
= &  \frac{\partial}{\partial x}
\Biggl( \frac{\partial}{\partial
  x}\int_{z_{\alpha-1/2}}^{z_{\alpha+1/2}} z\Sigma_{zx} dz +  z_{\alpha+1/2}\frac{\partial}{\partial x}\int_{z_{\alpha+1/2}}^\eta
\Sigma_{zx} dz\\
& -  z_{\alpha-1/2}\frac{\partial}{\partial x}\int_{z_{\alpha-1/2}}^\eta
\Sigma_{zx} dz\Biggr),
\end{align*}
and
\begin{eqnarray*}
& &\frac{\partial z_{\alpha+1/2}}{\partial x}
\left.\left(\frac{\partial}{\partial x} \int_z^\eta
\Sigma_{zx} dz_1 \right) \right|_{z_{\alpha+1/2}} = \frac{\partial z_{\alpha+1/2}}{\partial x}
\frac{\partial}{\partial x} \int_{z_{\alpha+1/2}}^\eta
\Sigma_{zx} dz + \left(\frac{\partial z_{\alpha+1/2}}{\partial x}\right)^2
\Sigma_{zx|_{\alpha+1/2}},\\
& &\frac{\partial z_{\alpha-1/2}}{\partial x}
\left.\left(\frac{\partial}{\partial x} \int_z^\eta
\Sigma_{zx} dz_1 \right) \right|_{z_{\alpha-1/2}} = \frac{\partial z_{\alpha-1/2}}{\partial x}
\frac{\partial}{\partial x} \int_{z_{\alpha-1/2}}^\eta
\Sigma_{zx} dz + \left(\frac{\partial z_{\alpha-1/2}}{\partial x}\right)^2
\Sigma_{zx|_{\alpha-1/2}}.
\end{eqnarray*}

\subsection{Definitions and closure relation}

The expression of the viscous terms generally involving second order derivatives,
their discretization requires quadrature formula that are
not inherited from the layer-averaged discretization. 
In particular, at this step of the paper, we adopt the following notations
\begin{equation}
\Sigma_{ab|_{\alpha+1/2}} \approx \Sigma_{ab,{\alpha+1/2}} \ ,
\label{eq:sigma_xz_inter}
\end{equation}
and
\begin{equation}
 \Sigma_{ab|_{\alpha}} \approx \Sigma_{ab,{\alpha}} \ ,
\label{eq:sigma_xz_m}
\end{equation}
and the following definitions,
\begin{equation}
\int_{z_{\alpha-1/2}}^{z_{\alpha+1/2}} \Sigma_{ab} dz \approx h_\alpha \Sigma_{ab,\alpha} \ ,
\label{eq:sigma_xz}
\end{equation}
with $(a,b) \in (x,z)^2$. For the terms having the form
$$\int_{z_{\alpha-1/2}}^{z_{\alpha+1/2}} z \Sigma_{ab} dz,$$
a closure relation is needed and we choose the approximation
\begin{equation}
\int_{z_{\alpha-1/2}}^{z_{\alpha+1/2}} z \Sigma_{ab} dz \approx \frac{z_{\alpha+1/2}^2 - z_{\alpha-1/2}^2}{2} \Sigma_{ab,\alpha} = h_\alpha z_\alpha \Sigma_{ab,\alpha} .
\label{eq:zsigma_xz}
\end{equation}

For each interface  $z_{\alpha+1/2}$ we introduce the unit normal vector  ${\bf n}_{\alpha+1/2}$
and the unit tangent vector ${\bf t}_{\alpha+1/2}$ given by:
$${\bf n}_{\alpha+1/2} = 
  \frac{1}{\sqrt{1 + \bigl(\frac{\partial z_{\alpha+1/2}}{\partial x}\bigr)^2}} 
   \left(\begin{array}{c} -\frac{\partial z_{\alpha+1/2}}{\partial
         x}\\ 1 \end{array} \right)\equiv \left(\begin{array}{c} -s_{\alpha+1/2} \\ c_{\alpha+1/2} \end{array} \right),\quad {\bf t}_{\alpha+1/2} = 
  \left(\begin{array}{c} c_{\alpha+1/2} \\ s_{\alpha+1/2} \end{array} \right).$$
 
Then, for $0\leqslant \alpha \leqslant N$, we have the following expression 
\begin{multline}
{\bf t}_{\alpha+1/2} \cdot \Sigma_{\alpha+1/2} {\bf n}_{\alpha+1/2}
= \frac{1}{1 + \bigl(\frac{\partial z _{\alpha+1/2}}{\partial
    x}\bigr)^2}\Biggl( \Sigma_{xz,\alpha+1/2} \\
- \frac{\partial
  z_{\alpha+1/2}}{\partial x}\Bigl( \Sigma_{xx,\alpha+1/2} + \frac{\partial
  z_{\alpha+1/2}}{\partial x}\Sigma_{zx,\alpha+1/2} - \Sigma_{zz,\alpha+1/2}\Bigr)\biggr),
\label{eq:BC_layer}
\end{multline}
{which can be rewritten as}
\begin{equation}
{\bf t}_{\alpha+1/2} \cdot \Sigma_{\alpha+1/2} {\bf
  n}_{\alpha+1/2} =  c_{\alpha+1/2}^2 \sigma_{\alpha+1/2} \ ,
\label{eq:BC_layer1}
\end{equation}
by introducing the following notation,
\begin{equation}
\sigma_{\alpha+1/2} = \Sigma_{xz,\alpha+1/2} - 
\frac{\partial
  z_{\alpha+1/2}}{\partial x}\Bigl( \Sigma_{xx,\alpha+1/2} + \frac{\partial
  z_{\alpha+1/2}}{\partial x}\Sigma_{zx,\alpha+1/2} -
\Sigma_{zz,\alpha+1/2}\Bigr).
\label{eq:petitsigma_xz_m}
\end{equation}
Remark that, for $0\leqslant\alpha\leqslant N$, the quantity ${\bf t}_{\alpha+1/2} \cdot \Sigma_{\alpha+1/2} {\bf
  n}_{\alpha+1/2}$ represents the tangential component of the stress
tensors at the interface $z_{\alpha+1/2}$.  And for $\alpha=\{0,N\}$, the
quantities~\eqref{eq:BC_layer} coincide with the boundary conditions
and hence are given. More precisely (since $c_{1/2}=c_b$) the Navier friction at bottom gives
\begin{equation}
 {\bf t}_{1/2} \cdot \Sigma_{1/2} {\bf
  n}_{1/2}= \frac{\kappa}{c_b} u_1 = \sigma_{1/2} c_{1/2}^2.
  \label{eq:coulom_layer1}
\end{equation}
Compared to equation (\ref{eq:BC_z_b1}), velocity in the first layer $u_1$ is used since $u_b$ is not a variable of our system. It is consistent with the convention
(\ref{eq:convention2}) and definition (\ref{eq:upwind_uT}).
At the surface we have
$${\bf t}_{N+1/2} \cdot \Sigma_{N+1/2} {\bf
  n}_{N+1/2}= \sigma_{N+1/2} c_{N+1/2}^2 = 0.$$

\begin{remark}
In (\ref{eq:coulom_layer1}) as in section \ref{sec:NS} , we use the expression ${\bf t}_{b} \cdot \Sigma{\bf n}_{b}$
to consider a Navier friction
at the bottom since on an impermeable boundary (\ref{eq:BC_z_b1}) is equivalent to (\ref{eq:BC_z_b}). For $1<\alpha<N-1$,
the flow can move across the interface
$z_{\alpha+1/2}$ and we cannot give a formulation directly comparable to (\ref{eq:BC_z_b}).
\end{remark}

\subsection{Layer-averaged Navier-Stokes system}

We have the following proposition.
\begin{proposition}
Using formula~\eqref{eq:sigma_xz},\eqref{eq:zsigma_xz} and \eqref{eq:petitsigma_xz_m}, the layer-averaging applied to the Navier-Stokes
system~\eqref{eq:NS_2d1_mod}-\eqref{eq:NS_2d2_mod} completed with the
boundary conditions~\eqref{eq:free_surf}-\eqref{eq:BC_z_b} leads to the system
\begin{eqnarray}
& & \frac{\partial}{\partial t} \sum_{j=1}^N h_j +
\frac{\partial}{\partial x} \sum_{j=1}^N h_j u_j = 0,\label{eq:NS_avz_111}\\
& & \frac{\partial}{\partial t} (h_\alpha u_\alpha) +
\frac{\partial}{\partial x}  \left( h_\alpha u_\alpha^2 + \frac{g}{2}h_\alpha H\right)
= -gh_{\alpha}\frac{\partial z_b}{\partial x} +
u_{\alpha+1/2}G_{\alpha+1/2} - u_{\alpha-1/2}
G_{\alpha-1/2}\nonumber\\
& & \qquad + \frac{\partial}{\partial x}\left(
  h_\alpha\Sigma_{xx,\alpha} - h_\alpha\Sigma_{zz,\alpha} +
  \frac{\partial}{\partial x} \Bigl( h_\alpha z_\alpha \Sigma_{zx,\alpha} \Bigr)\right) \nonumber\\
& & \qquad + z_{\alpha+1/2} \frac{\partial^2}{\partial x^2}
\sum_{j=\alpha+1}^N h_j\Sigma_{zx,j} - z_{\alpha-1/2}
\frac{\partial^2}{\partial x^2} \sum_{j=\alpha}^N h_j\Sigma_{zx,j} \nonumber\\
& & \qquad + {\sigma_{\alpha+1/2} - \sigma_{\alpha-1/2} } \qquad ,\label{eq:NS_avz_222}\\
& & w_\alpha = -\frac{1}{2}\frac{\partial (h_\alpha u_\alpha)}{\partial
  x} - \sum_{j=1}^{\alpha-1}
\frac{\partial (h_j u_j)}{\partial x}
+ u_\alpha \frac{\partial z_\alpha}{\partial x},\qquad \alpha=1,\dots,N \label{eq:NS_avz_444}
\end{eqnarray}
with the exchange terms $G_{\alpha\pm1/2}$ given
by~\eqref{eq:Qalphabis} {and the interface terms $\sigma_{\alpha\pm1/2}$ given by \eqref{eq:petitsigma_xz_m}}.

For smooth solutions, we obtain the balance
\begin{align}
&\frac{\partial }{\partial t} \left( \sum_{\alpha=1}^N
  E_{\alpha}\right) +  \frac{\partial}{\partial x}
\Biggl(\sum_{\alpha=1}^N u_\alpha\biggl( E_{\alpha}+
    \frac{g}{2} h_\alpha H - h_\alpha\bigl(\Sigma_{xx,\alpha} -
    \Sigma_{zz,\alpha}\bigr) \nonumber\\
& - \Bigl(\frac{\partial z_{\alpha} }{\partial x} h_\alpha \Sigma_{zx,\alpha}  + 
 h_{\alpha} \frac{\partial}{\partial x}\bigl(\frac 1 2 h_\alpha \Sigma_{zx,\alpha}+\sum_{j=\alpha+1}^N h_j\Sigma_{zx,j}\bigr)\Bigr) \biggr) 
 - \sum_{\alpha=1}^N  w_\alpha h_\alpha \Sigma_{zx,\alpha} \Biggr)  \nonumber\\
=& - \sum_{\alpha=1}^N  \biggl( \frac{\partial
    u_\alpha}{\partial x} h_\alpha \left( \Sigma_{xx,\alpha} -
  \Sigma_{zz,\alpha} \right) \nonumber\\
& +  \Bigl(\frac{\partial
  w_\alpha}{\partial x} + \frac{\partial z_\alpha}{\partial x}\frac{\partial
    u_\alpha}{\partial x}\Bigr)h_\alpha \Sigma_{zx,\alpha} 
    + { \sigma_{\alpha+1/2} \bigl(u_{\alpha+1} - u_{\alpha} \bigr) }
 \biggr) - \frac{\kappa}{c_b^3} u_1^2,
\label{eq:energy_mcc}
\end{align}
with $E_{\alpha}=\frac{ h_\alpha u_\alpha^2}{2}+\frac{ g  (z_{\alpha+1/2}^2 - z_{\alpha-1/2}^2)}{2}
= h_\alpha (\frac{ u_\alpha^2}{2} +g z_\alpha)$.

In~\eqref{eq:energy_mcc}, we use the convention
\begin{equation}
u_{0} = u_{1},\qquad  u_{N+1} = u_N.
\label{eq:convention2}
\end{equation}
\label{prop:mc_NS}
\end{proposition}

Before to give the proof of prop. \ref{prop:mc_NS}, we make few comments concerning the layer-averaging of the Cauchy stress tensor components.

\begin{remark}
Since the expression of the components of the Cauchy stress tensor are
not specified, we are not able to precise all the terms in
Eq.~\eqref{eq:energy_mcc} and we only intend to demonstrate that the
energy balance~\eqref{eq:energy_mcc} is consistent with~\eqref{eq:energy_eq_mod}. 
The nonnegativity of the right hand side of ~\eqref{eq:energy_mcc}
has then to be verified when specifying the rheological model (as it is done below in the Newtonian case).
\end{remark}

\begin{remark}
After injecting the definition \eqref{eq:petitsigma_xz_m} of $\sigma_{\alpha+1/2}$ 
in \eqref{eq:energy_mcc}, it appears that the following terms in the right hand side of~\eqref{eq:energy_mcc}
$$
- \sum_{\alpha=1}^N  \biggl( \frac{\partial
    u_\alpha}{\partial x} h_\alpha \left( \Sigma_{xx,\alpha} -
  \Sigma_{zz,\alpha} \right) 
  -  \frac{\partial  z_{\alpha+1/2}}{\partial x} \left( \Sigma_{xx,\alpha+1/2} - \Sigma_{zz,\alpha+1/2} \right)  
  \bigl(u_{\alpha+1} - u_{\alpha} \bigr) \biggr) 
$$
account for a layer-averaging of
$$- \int_{z_b}^\eta \frac{\partial u}{\partial x} ( \Sigma_{xx} -  \Sigma_{zz} ) dz,$$
appearing in the right hand side of ~\eqref{eq:energy_eq_mod}. 
Likewise, the term
\begin{equation}
-\int_{z_b}^\eta  \biggl( \frac{\partial u}{\partial z} \Sigma_{xz}
+   \frac{\partial w}{\partial x} \Sigma_{zx}  \biggr)  dz ,
\label{eq:energ_term1}
\end{equation}
in the right hand side of~\eqref{eq:energy_eq_mod} is discretized by
\begin{equation}
- \sum_{\alpha=1}^N \left(
\Sigma_{xz,\alpha+1/2} \bigl(u_{\alpha+1} - u_{\alpha} \bigr) 
+  h_\alpha \Sigma_{zx,\alpha}  \biggl(  \frac{\partial w_\alpha}{\partial x} + 
\frac{\partial z_\alpha}{\partial x}\frac{\partial u_\alpha}{\partial x} \biggr) 
- \biggl( \frac{\partial  z_{\alpha+1/2}}{\partial x} \biggr)^2 \Sigma_{zx,\alpha+1/2} \right)
\label{eq:energ_term2}
\end{equation}
in the layer-average context of Eq.~\eqref{eq:energy_mcc}.
A similar comparison can be done for the viscous terms involved in the left hand side 
of the two energy balances \eqref{eq:energy_eq_mod} and \eqref{eq:energy_mcc}.
\end{remark}

\begin{proof}[Proof of proposition~\ref{prop:mc_NS}]
The derivation of Eqs.~\eqref{eq:NS_avz_111} and~\eqref{eq:NS_avz_444}
is similar to what has been done to obtain the layer-averaged Euler system~\eqref{eq:H_bis}-\eqref{eq:eq6_bis}. 
Only the treatment of
the viscous terms $V_\alpha$ has to be specified.

Using the definitions~\eqref{eq:sigma_xz},\eqref{eq:zsigma_xz},
\eqref{eq:petitsigma_xz_m}, for $\alpha=\{1,N\}$ using the mimic of the boundary conditions 
it comes
\begin{eqnarray*}
V_\alpha 
& \approx & \frac{\partial}{\partial x}\left( h_\alpha\Sigma_{xx,\alpha} -
  h_\alpha\Sigma_{zz,\alpha} + \frac{\partial }{\partial x} \int_{z_{\alpha-1/2}}^{z_{\alpha+1/2}} z\Sigma_{zx} dz\right) \\
& & + z_{\alpha+1/2} \frac{\partial^2}{\partial x^2}
\sum_{j=\alpha+1}^N h_j\Sigma_{zx,j} - z_{\alpha-1/2}
\frac{\partial^2}{\partial x^2} \sum_{j=\alpha}^N h_j\Sigma_{zx,j}\\
& & { + \sigma_{\alpha+1/2} - \sigma_{\alpha-1/2}   }.
\end{eqnarray*}
The approximation \eqref{eq:zsigma_xz} gives
\begin{eqnarray*}
V_\alpha 
& \approx & R_\alpha + { \sigma_{\alpha+1/2} - \sigma_{\alpha-1/2}   } \\
& = & \frac{\partial}{\partial x}\left( h_\alpha\Sigma_{xx,\alpha} -
  h_\alpha\Sigma_{zz,\alpha} + \frac{\partial }{\partial x}
  \Bigl(\frac{z_{\alpha+1/2}^2 - z_{\alpha-1/2}^2}{2} \Sigma_{zx,\alpha}\Bigr)\right) \\
& & + z_{\alpha+1/2} \frac{\partial^2}{\partial x^2}
\sum_{j=\alpha+1}^N h_j\Sigma_{zx,j} - z_{\alpha-1/2}
\frac{\partial^2}{\partial x^2} \sum_{j=\alpha}^N h_j\Sigma_{zx,j}\\
& &  { + \sigma_{\alpha+1/2} - \sigma_{\alpha-1/2}   }.
\end{eqnarray*}
For the energy balance we write
\begin{eqnarray}
R_\alpha u_\alpha & = & \frac{\partial}{\partial x} \biggl( u_\alpha h_\alpha\Bigl( \Sigma_{xx,\alpha} - \Sigma_{zz,\alpha} \Bigr) +
u_\alpha \frac{\partial}{\partial x} \Bigl(\frac{z_{\alpha+1/2}^2 - z_{\alpha-1/2}^2}{2} \Sigma_{zx,\alpha}\Bigr) \nonumber\\
& &+ z_{\alpha+1/2} u_\alpha \frac{\partial}{\partial x}
\sum_{j=\alpha+1}^N h_j\Sigma_{zx,j} - z_{\alpha-1/2} u_\alpha \frac{\partial}{\partial x}
\sum_{j=\alpha}^N h_j\Sigma_{zx,j}\biggr)\nonumber\\
& &  - h_\alpha\Bigl( \Sigma_{xx,\alpha} - \Sigma_{zz,\alpha} \Bigr)
\frac{\partial u_\alpha}{\partial x} - \frac{\partial u_\alpha}{\partial x}  
\frac{\partial}{\partial x} \Bigl(\frac{z_{\alpha+1/2}^2 - z_{\alpha-1/2}^2}{2} \Sigma_{zx,\alpha}\Bigr) \nonumber\\
& &  - \frac{\partial}{\partial
  x} (z_{\alpha+1/2} u_\alpha) \frac{\partial}{\partial x}
\sum_{j=\alpha+1}^N h_j\Sigma_{zx,j} + \frac{\partial}{\partial
  x} (z_{\alpha-1/2} u_\alpha) \frac{\partial}{\partial x}
\sum_{j=\alpha}^N h_j\Sigma_{zx,j}.
\label{eq:r_alpha}
\end{eqnarray}
Notice that, using an integration by part, it comes that the three terms
\begin{multline*}
\frac{\partial}{\partial x} \biggl( 
u_\alpha \frac{\partial}{\partial x} \Bigl(\frac{z_{\alpha+1/2}^2 - z_{\alpha-1/2}^2}{2} \Sigma_{zx,\alpha}\Bigr) 
+ z_{\alpha+1/2} u_\alpha \frac{\partial}{\partial x}
\sum_{j=\alpha+1}^N h_j\Sigma_{zx,j} \\
- z_{\alpha-1/2} u_\alpha \frac{\partial}{\partial x}
\sum_{j=\alpha}^N h_j\Sigma_{zx,j}\biggr),
\end{multline*}
appearing in~Eq.~\eqref{eq:r_alpha} are a discretization of the
quantity
$$\frac{\partial}{\partial x} \biggl( u_\alpha\int_{z_{\alpha-1/2}}^{z_{\alpha+1/2}} \frac{\partial}{\partial
      x}\int_{z}^\eta \Sigma_{zx} dz_1 dz \biggr),$$
in the energy balance Eq.~\eqref{eq:energy_mcc}.

We can see that
\begin{align}
&\frac{\partial}{\partial x}  \biggl( u_\alpha\biggl( z_{\alpha+1/2}  \frac{\partial}{\partial x}
\sum_{j=\alpha+1}^N h_j\Sigma_{zx,j} - z_{\alpha-1/2} \frac{\partial}{\partial x}
\sum_{j=\alpha}^N h_j\Sigma_{zx,j}\biggr) \biggr)\nonumber\\
&= \frac{\partial}{\partial x} \biggl(  u_\alpha\biggl( (h_\alpha +z_{\alpha-1/2})  \frac{\partial}{\partial x}
\sum_{j=\alpha+1}^N h_j\Sigma_{zx,j} - z_{\alpha-1/2} \frac{\partial}{\partial x}
\sum_{j=\alpha+1}^N h_j\Sigma_{zx,j}- z_{\alpha-1/2} \frac{\partial}{\partial x}
 (h_\alpha\Sigma_{zx,\alpha })\biggr) \biggr)\nonumber\\
 &= \frac{\partial}{\partial x}  \biggl( u_\alpha\biggl( h_\alpha   \frac{\partial}{\partial x}
\sum_{j=\alpha+1}^N h_j\Sigma_{zx,j}- z_{\alpha-1/2} \frac{\partial}{\partial x}
( h_\alpha\Sigma_{zx,\alpha })\biggr)\biggr) \nonumber\\
 &= \frac{\partial}{\partial x}  \biggl( u_\alpha\biggl( h_\alpha   \frac{\partial}{\partial x} \biggl(
\sum_{j=\alpha+1}^N h_j\Sigma_{zx,j} + \frac{h_\alpha}{2} \Sigma_{zx,\alpha }\biggr)
- z_{\alpha} \frac{\partial}{\partial x}
( h_\alpha\Sigma_{zx,\alpha })\biggr)\biggr);
\end{align}
and
\begin{equation}
\frac{\partial}{\partial x}  \biggl(u_\alpha \frac{\partial}{\partial x} 
\Bigl(\frac{z_{\alpha+1/2}^2 - z_{\alpha-1/2}^2}{2} \Sigma_{zx,\alpha}\Bigr)\biggr) \\
= \frac{\partial}{\partial x}  \biggl(u_\alpha z_\alpha \frac{\partial}{\partial x}\Bigl(  h_\alpha \Sigma_{zx,\alpha } \Bigr)
+ u_\alpha h_\alpha \Sigma_{zx,\alpha}  \frac{\partial z_\alpha}{\partial x} \biggr) 
\end{equation}

Denoting $\tilde R_\alpha u_\alpha$ the last three terms
in Eq.~\eqref{eq:r_alpha}, we write
\begin{eqnarray*}
\tilde R_\alpha u_\alpha & = & - \frac{\partial}{\partial x}\biggl( 
\frac{z_{\alpha+1/2}^2 - z_{\alpha-1/2}^2}{2}\frac{\partial u_\alpha}{\partial x}  
\Sigma_{zx,\alpha}\biggr)  +
\frac{z_{\alpha+1/2}^2 - z_{\alpha-1/2}^2}{2} \frac{\partial^2 u_\alpha}{\partial x^2}\Sigma_{zx,\alpha}\\
& & - \frac{\partial}{\partial
  x} (z_{\alpha+1/2} u_\alpha) \frac{\partial}{\partial x}
\sum_{j=\alpha+1}^N h_j\Sigma_{zx,j} + \frac{\partial}{\partial
  x} (z_{\alpha-1/2} u_\alpha) \frac{\partial}{\partial x}
\sum_{j=\alpha}^N h_j\Sigma_{zx,j}\\
& = & \frac{\partial}{\partial x}\biggl( 
\bigl( h_\alpha w_\alpha - h_\alpha k_\alpha\bigr)
\Sigma_{zx,\alpha}\biggr)  +
\frac{z_{\alpha+1/2}^2 - z_{\alpha-1/2}^2}{2} \frac{\partial^2 u_\alpha}{\partial x^2}\Sigma_{zx,\alpha}\\
& & - \frac{\partial}{\partial
  x} (z_{\alpha+1/2} u_\alpha) \frac{\partial}{\partial x}
\sum_{j=\alpha+1}^N h_j\Sigma_{zx,j} + \frac{\partial}{\partial
  x} (z_{\alpha-1/2} u_\alpha) \frac{\partial}{\partial x}
\sum_{j=\alpha}^N h_j\Sigma_{zx,j},
\end{eqnarray*}
where~\eqref{eq:mean_w} has been used. And simple manipulations give
\begin{eqnarray*}
\tilde R_\alpha u_\alpha & = & \frac{\partial}{\partial x}\biggl(  w_\alpha h_\alpha
\Sigma_{zx,\alpha}\biggr) -  \biggl( \frac{\partial w_\alpha} {\partial
  x} + \frac{\partial z_\alpha}{\partial x}\frac{\partial
    u_\alpha}{\partial x}\biggr) h_\alpha\Sigma_{zx,\alpha} - 
k_\alpha \frac{\partial} {\partial x}(h_\alpha\Sigma_{zx,\alpha})\\
& & - \frac{\partial}{\partial
  x} (z_{\alpha+1/2} u_\alpha) \frac{\partial}{\partial x}
\sum_{j=\alpha+1}^N h_j\Sigma_{zx,j} + \frac{\partial}{\partial
  x} (z_{\alpha-1/2} u_\alpha) \frac{\partial}{\partial x}
\sum_{j=\alpha}^N h_j\Sigma_{zx,j}\\
& = & \frac{\partial}{\partial x}\biggl(  w_\alpha h_\alpha
\Sigma_{zx,\alpha}\biggr) - \Bigl( \frac{\partial w_\alpha} {\partial
  x} + \frac{\partial z_\alpha}{\partial x}\frac{\partial
    u_\alpha}{\partial x} \Bigr)  h_\alpha \Sigma_{zx,\alpha} \\
& & + \tilde{w}_{\alpha+1/2} \frac{\partial}{\partial x}
\sum_{j=\alpha+1}^N h_j\Sigma_{zx,j} - \tilde{w}_{\alpha-1/2}  \frac{\partial}{\partial x}
\sum_{j=\alpha}^N h_j\Sigma_{zx,j},
\end{eqnarray*}
with $\tilde{w}_{\alpha+1/2}$ defined by
$$\tilde{w}_{\alpha+1/2} = k_\alpha- \frac{\partial ( z_{\alpha+1/2}
  u_\alpha)}{\partial x}=k_{\alpha+1}- \frac{\partial ( z_{\alpha+1/2}
  u_{\alpha+1})}{\partial x}.$$
The two last terms of $\tilde R_\alpha u_\alpha$ give a telescoping series and vanish when summing since $\tilde{w}_{1/2} =0$ 
and $\sum_{j=\alpha+1}^N h_j\Sigma_{zx,j}$ vanish when $\alpha=N$.

Finally, the quantity
$$\sum_{\alpha=1}^N V_\alpha u_\alpha,$$
gives the expression involving of the terms related to the Cauchy
stress tensor in~\eqref{eq:energy_mcc} proving the result.
\end{proof}

\subsection{Newtonian fluids}

When considering a Newtonian fluid, the chosen form of the viscosity tensor is
\begin{eqnarray}
&\displaystyle \Sigma_{xx} = 2 \mu \frac{\partial u}{\partial x},   &\Sigma_{xz} = 
\mu \bigl( \frac{\partial u}{\partial z} + \frac{\partial w}{\partial x} \bigr),\label{eq:visco1}\\ 
&\displaystyle \Sigma_{zz} = 2 \mu \frac{\partial w}{\partial z},
&\Sigma_{zx} =\mu \bigl( \frac{\partial u}{\partial z} + \frac{\partial w}{\partial x}\bigr), 
\label{eq:visco2}
\end{eqnarray}
where $\mu$ is a dynamic  viscosity coefficient.

When considering the fluid rheology is given by~\eqref{eq:visco1}-\eqref{eq:visco2},
thus leading to $\Sigma_{zz} = - \Sigma_{xx}$ and $\Sigma_{xz} = \Sigma_{zx}$, 
prop.~\ref{prop:mc_NS} becomes:
\begin{lemma}
The layer-averaging applied to the Navier-Stokes system for a
newtonian fluid gives
\begin{eqnarray}
& & \frac{\partial}{\partial t} \sum_{j=1}^N h_j +
\frac{\partial}{\partial x} \sum_{j=1}^N h_j u_j = 0,\label{eq:NS_newt_1}\\
& & \frac{\partial}{\partial t} (h_\alpha u_\alpha) +
\frac{\partial}{\partial x}  \left( h_\alpha u_\alpha^2 + \frac{g}{2}h_\alpha H\right)
= -gh_{\alpha}\frac{\partial z_b}{\partial x} +
u_{\alpha+1/2}G_{\alpha+1/2} - u_{\alpha-1/2}
G_{\alpha-1/2}\nonumber\\
& & \qquad + \frac{\partial}{\partial x}\left(
  2 h_\alpha  \Sigma_{xx,\alpha}  +
  \frac{\partial}{\partial x} \Bigl( h_\alpha z_\alpha \Sigma_{zx,\alpha} \Bigr)\right) \nonumber\\
& & \qquad + z_{\alpha+1/2} \frac{\partial^2}{\partial x^2}
\sum_{j=\alpha+1}^N h_j\Sigma_{zx,j} - z_{\alpha-1/2}
\frac{\partial^2}{\partial x^2} \sum_{j=\alpha}^N h_j\Sigma_{zx,j} \nonumber\\
& & \qquad + {\sigma_{\alpha+1/2} - \sigma_{\alpha-1/2} } \qquad ,\label{eq:NS_newt_2}\\
& & w_\alpha = -\frac{1}{2}\frac{\partial (h_\alpha u_\alpha)}{\partial
  x} - \sum_{j=1}^{\alpha-1}
\frac{\partial (h_j u_j)}{\partial x}
+ u_\alpha \frac{\partial z_\alpha}{\partial x},\qquad \alpha=1,\dots,N \label{eq:NS_newt_3}
\end{eqnarray}
where exchange terms $G_{\alpha\pm1/2}$ are still given by~\eqref{eq:Qalphabis}
and the interface terms $\sigma_{\alpha\pm1/2}$ defined by \eqref{eq:petitsigma_xz_m}
are here reduced to
\begin{eqnarray}
\sigma_{\alpha+1/2} & = & 
 - 2 \Sigma_{xx,\alpha+1/2} \frac{\partial  z_{\alpha+1/2}}{\partial x}
 + \Sigma_{zx,\alpha+1/2} \left(   1 - \Bigl( \frac{\partial z_{\alpha+1/2}}{\partial x}\Bigr)^2 \right).
\label{eq:sigma_newt}
\end{eqnarray}
For smooth solutions, we obtain the balance
\begin{align}
&\frac{\partial }{\partial t} \left( \sum_{\alpha=1}^N
  E_{\alpha}\right) +  \frac{\partial}{\partial x}
\Biggl(\sum_{\alpha=1}^N u_\alpha\biggl( E_{\alpha}+
    \frac{g}{2} h_\alpha H - 2 h_\alpha \Sigma_{xx,\alpha} \nonumber\\
& - \Bigl(\frac{\partial z_{\alpha} }{\partial x} h_\alpha \Sigma_{zx,\alpha}  + 
 h_{\alpha} \frac{\partial}{\partial x}\bigl(\frac 1 2 h_\alpha \Sigma_{zx,\alpha}+\sum_{j=\alpha+1}^N h_j\Sigma_{zx,j}\bigr)\Bigr) \biggr) 
 - \sum_{\alpha=1}^N  w_\alpha h_\alpha \Sigma_{zx,\alpha} \Biggr)  \nonumber\\
=& - \sum_{\alpha=1}^N  \biggl( \frac{\partial
    u_\alpha}{\partial x} 2 h_\alpha \Sigma_{xx,\alpha}  
    +  \Bigl(\frac{\partial
  w_\alpha}{\partial x} + \frac{\partial z_\alpha}{\partial x}\frac{\partial
    u_\alpha}{\partial x}\Bigr)h_\alpha \Sigma_{zx,\alpha} 
    + { \sigma_{\alpha+1/2} \bigl(u_{\alpha+1} - u_{\alpha} \bigr) }
 \biggr) - \frac{\kappa}{c_b^3} u_1^2,
\label{eq:energy_mcc_newt}
\end{align}
\label{lem:mcc_newt} \end{lemma}

If we look at the energy balance for the continuous setting \eqref{eq:energy_eq_mod},
we have, by using \eqref{eq:visco1}-\eqref{eq:visco2}, the following non-positive right hand side,
\begin{equation}
- \int_{z_b}^\eta \dfrac{1}{\mu}  \Bigl( \Sigma_{xx}^2 + \Sigma_{zx}^2 \Bigr) dz - \frac{\kappa}{c_b^3} u_b^2,
\label{eq:energy_newt_RHS}
\end{equation}
whereas, after including \eqref{eq:sigma_newt} in \eqref{eq:energy_mcc_newt}, the right hand side of the discrete energy balance
of the layer-averaged model leads to
\begin{align}
R_E=& - \sum_{\alpha=1}^N  \biggl( ~ 
2 \frac{\partial u_\alpha}{\partial x}  h_\alpha \Sigma_{xx,\alpha} 
- 2 \Sigma_{xx,\alpha+1/2} (u_{\alpha+1} - u_{\alpha}) \frac{\partial  z_{\alpha+1/2}}{\partial x} 
+ \Bigl(\frac{\partial  w_\alpha}{\partial x} + 
\frac{\partial z_\alpha}{\partial x}\frac{\partial u_\alpha}{\partial x}\Bigr)h_\alpha \Sigma_{zx,\alpha}   \nonumber\\
&  +  \Sigma_{zx,\alpha+1/2} (u_{\alpha+1} - u_{\alpha})  
\left(   1 - \Bigl( \frac{\partial z_{\alpha+1/2}}{\partial x}\Bigr)^2 \right)   ~ \biggr) 
- \frac{\kappa}{c_b^3} u_1^2 \ .
\label{eq:energy_mcc_newt_RHSDvp}
\end{align}
The aim of the next proposition is to mimic \eqref{eq:energy_mcc_newt_RHSInterf}.

\begin{proposition}
The layer-averaging, given in lemma \ref{lem:mcc_newt}, is applied to the Navier-Stokes system for a
newtonian fluid with the following consistent expressions of the  rheology terms at the interface $\alpha+1/2$,
\begin{eqnarray}
h_{\alpha+1/2}  \Sigma_{xx,\alpha+1/2} & = &  - h_{\alpha+1/2}  \Sigma_{zz,\alpha+1/2}  \nonumber \\
& = & 2\mu  ~ \left( ~ \dfrac{1}{2}  \biggl(h_\alpha  \frac{\partial u_\alpha}{\partial x}+h_{\alpha+1}  \frac{\partial u_{\alpha+1}}{\partial x}\biggr) - 
  \frac{\partial z_{\alpha+1/2}}{\partial x}(u_{\alpha+1} - u_\alpha) ~ \right) \ ,
 \label{eq:sigma_xx_NewtonInter}
\end{eqnarray}
\begin{eqnarray}
h_{\alpha+1/2}  \Sigma_{zx,\alpha+1/2} & =  & h_{\alpha+1/2}  \Sigma_{xz,\alpha+1/2} \nonumber  \\
 & = &  \mu ~ \Biggl( ~ \dfrac{1}{2}  \biggl( h_\alpha (
\frac{\partial w_\alpha}{\partial x} +  \frac{\partial z_\alpha}{\partial x}\frac{\partial
    u_\alpha}{\partial x} ) + h_{\alpha+1}(\frac{\partial w_{\alpha+1}}{\partial x} +  \frac{\partial z_{\alpha+1}}{\partial x}\frac{\partial
    u_{\alpha+1} }{\partial x}) \biggr) \nonumber\\
    & &+ { (u_{\alpha+1} - u_\alpha) } 
    \left(   1 - \Bigl( \frac{\partial z_{\alpha+1/2}}{\partial x}\Bigr)^2 \right)    ~ \Biggr) \ .
\label{eq:sigma_zx_NewtonInter}
\end{eqnarray}
and, since the rheology terms are more related to elliptic than hyperbolic type behaviour, we used the centred approximation
for the  rheology terms at the layers $\alpha$,
\begin{equation}
\Sigma_{ab,\alpha} = \dfrac{\Sigma_{ab,{\alpha+1/2}} + \Sigma_{ab,{\alpha-1/2}} }{2},
\label{eq:sigma_ab_layer_moy}
\end{equation}
with $(a,b) \in (x,z)^2$. Then we obtain an energy inequality since the right hand side of the discrete energy balance
of the layer-averaged model, defined by \eqref{eq:energy_mcc_newt_RHSDvp}, leads here to
\begin{equation}
R_E= - \sum_{\alpha=0}^N \dfrac{h_{\alpha+1/2}}{\mu} \biggl(  \Sigma_{xx,\alpha+1/2}^2 + \Sigma_{zx,\alpha+1/2}^2 \biggr) - \frac{\kappa}{c_b^3} u_1^2 \ .
\label{eq:energy_mcc_newt_RHSInterf}
\end{equation}
\label{prop:newton_mcc}
\end{proposition}

\begin{proof}
The expression \eqref{eq:energy_mcc_newt_RHSInterf} clearly mimics the continuous one given by \eqref{eq:energy_newt_RHS}.
Moreover it is possible to exhibit a kind of consistency of the definitions \eqref{eq:energy_mcc_newt_RHSInterf}-\eqref{eq:sigma_xx_NewtonInter}.
Indeed if we express the derivatives of the newtonian stress terms along the interface ${\alpha+1/2}$, on one hand, we have 
\begin{multline*}
 {\Sigma_{xx}}_{|z=z_{{\alpha+1/2}}(x,t)}  = 2\mu \ \partial_x u(x,z,t)_{|z = z_{{\alpha+1/2}}(x,t)} \\ = 
  2\mu \left( \frac {\partial u(x,z_{{\alpha+1/2}}(x,t),t) } {\partial x} - 
  \frac{\partial z_{{\alpha+1/2}}(x,t)}{\partial x} \partial_z u(x,z,t)_{|z = z_{{\alpha+1/2}}(x,t)} \right),
 \end{multline*}
which is consistent with \eqref{eq:sigma_xx_NewtonInter}. And, on the other hand,  
we have, $${\Sigma_{zx}}_{|z=z_{{\alpha+1/2}}(x,t)}=\mu \left( \partial_z u(x,z,t)_{|z=z_{{\alpha+1/2}}(x,t)} + 
\partial_x w(x,z,t)_{|z=z_{{\alpha+1/2}}(x,t)} \right).$$ 
Additionally, we can write
$$
 \partial_x w(x,z,t)_{|z = z_{{\alpha+1/2}}(x,t)}  = \frac {\partial w(x,z_{{\alpha+1/2}}(x,t),t) } {\partial x} - 
 \frac{\partial z_{{\alpha+1/2}}(x,t)}{\partial x} \partial_z w(x,z,t)_{|z = z_{{\alpha+1/2}}(x,t)} \ ,
$$
and, using the incompressibility condition, we get,
$$\partial_z w(x,z,t)_{|z = z_{{\alpha+1/2}}(x,t)}
= -\partial_x u(x,z,t)_{|z = z_{{\alpha+1/2}}(x,t)} \ . $$ 
Therefore we have,
\begin{multline*}
 \partial_x w(x,z,t)_{|z = z_{{\alpha+1/2}}(x,t)}  = \frac {\partial w(x,z_{{\alpha+1/2}}(x,t),t) } {\partial x} + \\
 \frac{\partial z_{{\alpha+1/2}}(x,t)}{\partial x} \left(\frac {\partial u(x,z_{{\alpha+1/2}}(x,t),t) } {\partial x} 
 - \frac{\partial z_{{\alpha+1/2}}(x,t)}{\partial x} \partial_z u(x,z,t)_{|z = z_{{\alpha+1/2}}(x,t)} \right).
 \end{multline*}
Finally, this leads to the following expression
\begin{multline*}
 {\Sigma_{zx}}_{|z=z_{{\alpha+1/2}}(x,t)}=\mu \Bigl(  \frac {\partial w(x,z_{{\alpha+1/2}}(x,t),t) } {\partial x} +  
 \frac{\partial z_{{\alpha+1/2}}(x,t)}{\partial x} \frac {\partial u(x,z_{{\alpha+1/2}}(x,t),t) } {\partial x} + \\
\Bigl( 1 -  \frac{\partial z_{{\alpha+1/2}}(x,t)}{\partial x}^2 \Bigr) \partial_z u(x,z,t)_{|z = z_{{\alpha+1/2}}(x,t)}  \Bigr),
\end{multline*}
which is consistent with \eqref{eq:sigma_zx_NewtonInter}.

The energy inequality is obtain by injecting \eqref{eq:sigma_xx_NewtonInter}, \eqref{eq:sigma_zx_NewtonInter} and
\eqref{eq:sigma_ab_layer_moy} in \eqref{eq:energy_mcc_newt_RHSDvp}.
\end{proof}

\begin{remark}
We can remark in the lemma \eqref{lem:mcc_newt} that the rheology terms are both at the interface and in the layers. 
Thus an other strategy could be to defined them at the layer, and to average the terms at the interface. 
In this case, we have
\begin{eqnarray}
h_\alpha \Sigma_{xx,\alpha} & = &  - h_\alpha \Sigma_{zz,\alpha} \nonumber \\
& = & 2\mu  ~ \left( ~ h_\alpha  \frac{\partial u_\alpha}{\partial x} - 
  \Bigl(\frac{\partial z_{\alpha+1/2}}{\partial x}\frac {u_{\alpha+1} - u_\alpha}{2} +
 \frac{\partial z_{\alpha-1/2}}{\partial x} \frac {u_{\alpha} - u_{\alpha-1}}{2} \Bigr) ~ \right)
 \label{eq:sigma_xx_NewtonLayer}
\end{eqnarray}
\begin{eqnarray}
 h_\alpha \Sigma_{zx,\alpha} & =  & h_\alpha \Sigma_{xz,\alpha} \nonumber  \\
 & = &\mu ~ \Biggl( ~ h_\alpha
\frac{\partial w_\alpha}{\partial x} + h_\alpha \frac{\partial z_\alpha}{\partial x}\frac{\partial
    u_\alpha}{\partial x} + \frac{u_{\alpha+1} - u_{\alpha}}{2}  
    \left(   1 - \Bigl( \frac{\partial z_{\alpha+1/2}}{\partial x}\Bigr)^2 \right)   \nonumber  \\
 & & \quad \quad +  \frac{u_{\alpha} - u_{\alpha-1}}{2}  
    \left(   1 - \Bigl( \frac{\partial z_{\alpha-1/2}}{\partial x}\Bigr)^2 \right)  ~ \Biggr) ,
\label{eq:sigma_zx_NewtonLayer}
\end{eqnarray}
which are also consistent expressions of the tensor, and the following averaging is introduced,
\begin{equation}
\Sigma_{ab,\alpha+1/2} = \dfrac{\Sigma_{ab,{\alpha+1}} + \Sigma_{ab,{\alpha}} }{2},
\label{eq:sigma_ab_inter_moy}
\end{equation}
 and leads to
an energy inequality, since the right hand side of the discrete energy balance
of the layer-averaged model, defined by \eqref{eq:energy_mcc_newt_RHSDvp}, leads here to
\begin{equation}
R_E = - \sum_{\alpha=1}^N \dfrac{h_{\alpha}}{\mu} \biggl(  \Sigma_{xx,\alpha}^2 + \Sigma_{zx,\alpha}^2 \biggr) 
- \frac{\kappa}{c_b^3} u_1^2 \ .
\label{eq:energy_mcc_newt_RHSLayer}
\end{equation}
This strategy seems to be more natural since, in the spirit of the layer-averaged model,
the unknowns are mainly localised in the layers. However the main drawback is the stencil 
of the interface rheology terms which are not compact. For instance, the term
$\Sigma_{xx,\alpha+1/2}$ will be expressed in function of $u_{\alpha+2}, u_{\alpha+1}$ and $u_{\alpha-1}$.
\end{remark}

\subsection{An extended Saint-Venant system}

In the simplified case of a single layer, the model given in
prop.~\ref{prop:mc_NS} corresponds to the classical Saint-Venant system but
completed with rheology terms.
\begin{proposition}
The classical Saint-Venant corresponds to the single-layer version of
the layer-averaged Navier-Stokes system. With obvious notations,
it is given by
\begin{eqnarray*}
& & \frac{\partial H}{\partial t} +
\frac{\partial}{\partial x} (H \overline{u}) = 0,\label{eq:SV_1}\\
& & \frac{\partial (H\overline{u})}{\partial t} +
\frac{\partial}{\partial x}  \left( H \overline{u}^2 + \frac{g}{2}H^2\right)
= -gH\frac{\partial z_b}{\partial x} \nonumber\\
& & + \frac{\partial}{\partial x}\left(
  H\overline{\Sigma}_{xx} - H\overline{\Sigma}_{zz} +
  \frac{\partial}{\partial x} \Bigl(\frac{(H+z_b)^2 - z_{b}^2}{2} \overline{\Sigma}_{zx} \Bigr)\right)  - z_{b}
\frac{\partial^2}{\partial x^2} (H \overline{\Sigma}_{zx}) - \frac{\kappa}{c_b^3} \overline{u},\label{eq:SV_2}\\
& & \overline{w} = -\frac{1}{2}\frac{\partial (H\overline{u})}{\partial
  x}  + \overline{u} \frac{\partial}{\partial x} \left( \frac{H+2z_b}{2}\right).\label{eq:SV_3}
\end{eqnarray*}
For smooth solutions, we obtain the balance
\begin{multline*}
\frac{\partial E }{\partial t} +  \frac{\partial}{\partial x}
\Biggl(\overline{u}\biggl( E +
    \frac{g}{2} H^2 - H\bigl(\overline{\Sigma}_{xx} -
    \overline{\Sigma}_{zz}\bigr) 
 - \frac{\partial}{\partial
      x}\Bigl( \frac{\partial (H+2z_b)}{\partial x}\overline{\Sigma}_{xz} + 
      \frac{H}{2} \frac{\partial}{\partial x} (H\overline{\Sigma}_{xz}) \Bigr) \biggr) \\
 - H\overline{w}
  \overline{\Sigma}_{zx} \Biggr) = - H\frac{\partial
    \overline{u}}{\partial x} \left( \overline{\Sigma}_{xx} -
  \overline{\Sigma}_{zz} \right) 
 - H \Bigl(\frac{\partial
  \overline{w}}{\partial x} + \frac{1}{2}\frac{\partial (H+2z_b)}{\partial
    x}\frac{\partial \overline{u}}{\partial x}\Bigr)\overline{\Sigma}_{zx} - \frac{\kappa}{c_b^3} \overline{u}^2,
\end{multline*}
with $E =\frac{ H \overline{u}^2}{2}+\frac{g}{2}
  \Bigl( (H+z_b)^2 - z_{b}^2\Bigr)$.
\label{prop:mc_SV}
\end{proposition}
In the particular case of a newtonian fluid, the Saint-Venant system
given in prop.~\ref{prop:mc_SV} reduces to
\begin{eqnarray}
& & \frac{\partial H}{\partial t} +
\frac{\partial}{\partial x} (H \overline{u}) = 0,\label{eq:SV_11}\\
& & \frac{\partial (H\overline{u})}{\partial t} +
\frac{\partial}{\partial x}  \left( H \overline{u}^2 + \frac{g}{2}H^2\right)
= -gH\frac{\partial z_b}{\partial x} \nonumber\\
& & \qquad + \frac{\partial}{\partial x}\left(
  4\mu H\frac{\partial \overline{u}}{\partial x} +
  \frac{\partial}{\partial x} \mu\Bigl(\frac{(H+z_b)^2 - z_{b}^2}{2} 
  \frac{\partial \overline{w}}{\partial x} \Bigr)\right) - z_{b} \mu
\frac{\partial^2}{\partial x^2} \Bigl( H \frac{\partial \overline{w}}{\partial x} \Bigr)  
- \frac{\kappa}{c_b^3}\overline{u},\label{eq:SV_22}\\
& & \overline{w} = -\frac{1}{2}\frac{\partial (H\overline{u})}{\partial
  x}  + \overline{u} \frac{\partial}{\partial x} \left( \frac{H+2z_b}{2}\right).\label{eq:SV_33}
\end{eqnarray}
For smooth solutions, we obtain the energy balance
\begin{multline}
\frac{\partial E }{\partial t} +  \frac{\partial}{\partial x}
\Biggl(\overline{u}\biggl( E +
    \frac{g}{2} H^2 - 4\mu H\frac{\partial \overline{u}}{\partial x}\bigr)  
    - \frac{\partial}{\partial       x}\Bigl( \mu\bigl(\frac{\partial (H+2z_b)}{\partial x}
    \frac{\partial \overline{w}}{\partial x} + \frac{H}{2} \frac{\partial}{\partial x} 
    (H\frac{\partial \overline{w}}{\partial x}) \bigr)\Bigr) \biggr) \\
  - \mu\frac{H}{2}\frac{\partial \overline{w}^2}{\partial x}\Biggr) = -
 4\mu H\left(\frac{\partial
    \overline{u}}{\partial x}\right)^2 
- \mu H \left(\frac{\partial
  \overline{w}}{\partial x} + \frac{1}{2}\frac{\partial (H+2z_b)}{\partial
    x}\frac{\partial \overline{u}}{\partial x}\right)^2 - \frac{\kappa}{c_b^3} \overline{u}^2.
\label{eq:eSV11}
\end{multline}
\begin{remark}
Notice that, compared to the classical viscous Saint-Venant system~\cite{gerbeau}, the model~\eqref{eq:SV_11}-\eqref{eq:eSV11} has
complementary terms.
\end{remark}

\section{Conclusion}

We have proposed a layer-averaged discretization for the approximation
of the incompressible free surface Euler and Navier-Stokes
equations. The obtained models do not rely on any asymptotic expansion
but on a criterion of minimal kinetic energy. Notice also that the
layer averaging for the Navier-Stokes system has been carried out for
a fluid with a general rheology.

Since these models are formulated over a fixed domain, it is
possible to derive efficient numerical techniques for their
approximation. For the approximation of the proposed models, a finite volume strategy~-- relying on a kinetic
interpretation and satisfying stability properties such as a fully
discrete entropy inequality~-- will be published in a forthcoming paper.

\section{Acknowledgement}

The work presented in this paper was supported in part by the Inria Project Lab ``Algae in Silico'' 
and the CNRS-INSU, TelluS-INSMI-MI program, project CORSURF.
It was realised during the secondment of the third author in the Ange Inria team.

\bibliographystyle{amsplain}
\bibliography{boussinesq}

\end{document}